\documentclass[11pt,a4paper,reqno]{amsart}
\usepackage[applemac]{inputenc}
\usepackage[T1]{fontenc}
\usepackage{amsmath}
\usepackage{amsthm}
\usepackage{amsfonts}
\usepackage{amssymb}
\usepackage{graphicx}
\usepackage{amsbsy}
\usepackage{mathrsfs}
\usepackage{bbm}
\newtheorem{theorem}{Theorem}[section]
\newtheorem{lemma}[theorem]{Lemma}

\newtheorem{proposition}[theorem]{Proposition}

\theoremstyle{remark}
\newtheorem{remark}[theorem]{\it \bf{Remark}\/}

\numberwithin{equation}{section}
\catcode`@=11
\def\section{\@startsection{section}{1}%
  \z@{1.5\linespacing\@plus\linespacing}{.5\linespacing}%
  {\normalfont\bfseries\large\centering}}
\catcode`@=12
\newcommand{\be}{\begin{equation}}
\newcommand{\ee}{\end{equation}}
\newcommand{\bea}{\begin{eqnarray}}
\newcommand{\eea}{\end{eqnarray}}
\newcommand{\bee}{\begin{eqnarray*}}
\newcommand{\eee}{\end{eqnarray*}}

\def\pa{\partial}
\def\pr{\partial}

\def\RR{\mathbb{R}}

\def\de{\delta}

\def\ep{\varepsilon}

\def\fref#1{{\rm (\ref{#1})}}

\catcode`@=11
\def\supess{\mathop{\operator@font Sup\,ess}}
\catcode`@=12

\def\RR{\mathbb{R}}

\def\e{\varepsilon}

\def\fref#1{{\rm (\ref{#1})}}

\def\R2+{\RR ^2_+}

\def\lsl{\frac{\lambda_s}{\lambda}}

\def\pa{\partial}

\def\lim{\mathop{\rm lim}}

\def\sup{\mathop{\rm sup}}

\def\l{\lambda}

\def\log{{\rm log}}

\def\lsl{\frac{\lambda_s}{\lambda}}

\def\Phit{\widetilde{\Phi}}

\def\pa{\partial}

\def\Lamdba{\Lambda}

\def\pa{\partial}
\def\la{\langle}
\def\matchal{\mathcal}
\def\ra{\rangle}
\def\Mod{\textrm{Mod}}
\def\L{\mathcal L}
\def\Psit{\widetilde{\Psi}}

\begin{document}

\title[]{On strongly anisotropic type I blow up}
\author[F.Merle]{Frank Merle}
\address{LAGA, Universit\'e de Cergy Pontoise, France and IHES}
\email{merle@math.u-cergy.fr}
\author[P. Rapha\"el]{Pierre Rapha\"el}
\address{Laboratoire J.A. Dieudonn\'e, Universit\'e de Nice-Sophia Antipolis, France}
\email{praphael@unice.fr}
\author[J. Szeftel]{Jeremie Szeftel}
\address{Laboratoire Jacques-Louis Lions, Universit\'e Paris 6, France}
\email{jeremie.szeftel@upmc.fr}

\begin{abstract} 
We consider the energy super critical 4 dimensional semilinear heat equation $$\pa_tu=\Delta u+|u|^{p-1}u, \ \ x\in \Bbb R^4, \ \ p>5.$$ Let $\Phi(r)$ be a {\it three} dimensional radial self similar solution for the  {\it three} supercritical probmem as exhibited and studied in \cite{CRS}. We show the finite codimensional transversal stability of the corresponding blow up solution by exhibiting a manifold of  finite energy blow up solutions of the {\it four} dimensional problem with cylindrical symmetry which blows up as 
$$u(t,x)\sim \frac{1}{(T-t)^{\frac{1}{p-1}}}U(t,Y),  \ \ Y=\frac{x}{\sqrt{T-t}}$$ 
with the profile $U$ given to leading order by 
$$U(t,Y)\sim\frac{1}{(1+b(t)z^2)^{\frac 1{p-1}}}\Phi\left(\frac{r}{\sqrt{1+b(t)z^2}}\right), \ \ Y=(r,z), \ \ b(t)=\frac{c}{|\log(T-t)|}$$ 
corresponding to a constant profile $\Phi(r)$ in the $z$ direction reconnected to zero along the moving free boundary $|z(t)|\sim \frac{1}{\sqrt{b}}\sim \sqrt{|\log (T-t)|}.$ Our analysis revisits the stability analysis of the self similar ODE blow up \cite{BK, MZduke,MZgaffa} and combines it with the study of the Type I self similar blow up \cite{CRS}. This provides a robust canonical framework for the construction of strongly anisotropic blow up bubbles.
\end{abstract}

\maketitle



\section{Introduction}



\subsection{Type I and type II blow up}


Let us consider the focusing  nonlinear heat equation 
\begin{equation}\label{eq:heat}
\left\{\begin{array}{l}
\pr_tu = \Delta u + |u|^{p-1}u,\,\,\,\, (t,x)\in \mathbb{R}\times\mathbb{R}^d,\\
u_{|_{t=0}}=u_0,
\end{array}\right.
\end{equation}
where $p>1$. This model dissipates the total energy
\bea
E(u)=\frac{1}{2}\int |\nabla u|^2-\frac{1}{p+1}\int |u|^{p+1}, \,\,\,\,\frac{dE}{dt}=-\int (\pr_tu)^2<0
\eea
and admits a scaling invariance: if $u(t,x)$ is a solution, then so is
\bea
u_\l( t ,x)=\l^{\frac{2}{p-1}}u(\l^2 t, \l x),\,\,\, \l>0.
\eea
This transformation is an isometry on the homogeneous Sobolev space
$$\|u_\l(t,\cdot)\|_{\dot{H}^{s_c}}=\|u( t,\cdot)\|_{\dot{H}^{s_c}}\textrm{ for }s_c=\frac{d}{2}-\frac{2}{p-1}.$$
We address in this paper the question of the existence and stability of blow up dynamics in the energy super critical range $s_c>1$ emerging from well localized initial data. There is an important literature devoted to the question of the description of blow up solutions for \eqref{eq:heat} and we recall some key facts related to our analysis.\\

\noindent{\em Type I ODE blow-up}. Type I singularities blow up with the self similar speed $$\|u(t,\cdot)\|_{L^{\infty}}\sim \frac{1}{(T-t)^{\frac 1{p-1}}}.$$ These solutions concentrate to leading order at a point 
$$u(t,x)\sim \frac{1}{\l(t)^{\frac{2}{p-1}}}v\left(\frac{x}{\l(t)}\right),  \ \ \l(t)= \sqrt{T-t},$$ 
where the blow up profile $v$ solves the non linear elliptic equation
\be
\label{ellipticequation}
\Delta v-\frac 1 2\left(\frac{2}{p-1}v+y\cdot\nabla v\right) +|v|^{p-1}v=0.
\ee 
The ODE blow up corresponds to the special solution to \eqref{ellipticequation} $$v=\left(\frac{1}{p-1}\right)^{\frac 1{p-1}},$$ and the existence and stability of the associated blow up dynamics has been studied in the series of papers  \cite{Gi,Gi1,Gi2,Gi3, BK, MZduke,MZgaffa}.\\

\noindent{\em Type I self similar blow-up}. There also exist radial solutions to \eqref{ellipticequation} which vanish at infinity. They correspond to the shooting problem
\be
\label{vnnevokenone}
\left|\begin{array}{lll}\displaystyle\Phi''+\frac{d-1}{r}\Phi'-\frac 1 2\left(\frac{2}{p-1}\Phi+r\Phi'\right) +\Phi^p=0,\\[2mm]
\Phi'(0)=0,\\[2mm]
\displaystyle\lim_{r\to+\infty}\Phi(r)=0.
\end{array}\right.
\ee
A countable class of such solutions has been constructed using either a direct Lyapunov functional approach \cite{lepin,troy,buddone,buddselfsim} or a bifurcation argument \cite{bizon,CRS}, and these solutions satisfy 
\be
\label{propertiesphi}
\left|\begin{array}{ll}\Phi(r)\gtrsim \frac{1}{\la r\ra^{\frac{2}{p-1}}}\\[4mm]
\Phi\in \mathcal C^\infty[0,+\infty), \ \ |\pa_r^k\Phi|\lesssim_k\frac{1}{\la r\ra^{\frac{2}{p-1}+k}}, \ \ k\in \Bbb N,
\end{array}\right.
\ee
where $\la r\ra=\sqrt{1+r^2}$. In particular, these solutions have infinite energy, but they can be shown to be the blow up profile for a finite codimensional class of finite energy smooth initial data, \cite{CRS}, see also \cite{DS}. The finite codimension of self similar blow up initial data is in one to one correspondance with the nonpositive eigenmodes of the linearized operator restricted to radial functions: $$\mathcal L_r=-\pa_{rr}-\frac{d-1}{r}\pa_r+\frac 12\left(\frac{2}{p-1}+r\pa_r\right)-p\Phi^{p-1}$$ which is self adjoint for the weighted $e^{-\frac{r^2}{4}}r^{d-1}dr$ measure:
\be
\label{spectralgapeigenvalueintro}
\lambda_{-\ell_0}<\dots<\lambda_{-1}=-1<0<\l_{0}<\l_{1}<\dots, \ \ \lim_{j\to +\infty}\l_j=+\infty.
\ee

\noindent{\em Type II blow-up}. Type II singularities are slower than self similar $$\lim_{t\to T}(T-t)^{\frac 1{p-1}}\|u(t,\cdot)\|_{L^{\infty}}=+\infty.$$  Such dynamics have been ruled out in the radial class for $p<p_{JL}$ in \cite{MaMe1,MaMe2} where $p_{JL}$ denotes the Joseph-Lundgren exponent
 \be
\label{exponentpjl}
p_{JL}:=\left\{\begin{array}{ll} +\infty\ \ \mbox{for}\ \ d\leq 10,\\
1+\frac{4}{d-4-2\sqrt{d-1}}\ \ \mbox{for}\ \ d\geq 11,
\end{array}\right.
\ee
and the result is sharp since type II blow up solutions can be constructed for $p>p_{JL}$, \cite{HV, MaMe1, Mizo, Collotheat} in connection with the general approach developed in \cite{RaphRod,MRR,RSc} for equation (\ref{eq:heat}) and the corresponding semilinear wave and Schrodinger equation for $p>p_{JL}$.


\subsection{Dimensional reduction and anisotropic blow up}


There exist a variety of blow up problems where the construction relies on a dimensional reduction and the use of lower dimensional soliton like solutions.\\
 
A typical example is the construction of ring solutions for the two dimensional non linear Schr\"odinger equation $$i\pa_tu+\Delta u+u|u|^{p-1}=0, \ \ x\in \Bbb R^2,  \ \ 3<p\leq 5$$ as discovered in \cite{Rring}, \cite{Fring}, see also \cite{RSring}, \cite{MRSring}. For example for $p=5$, these solutions concentrate in finite time on the unit sphere $$u(t,r)\sim\frac{1}{\l(t)^{\frac 12}}Q\left(\frac{r-1}{\l(t)}\right)e^{i\gamma(t)}$$ where $Q$ is the one dimensional ground state solitary wave and $\l(t)$ corresponds to the stable blow up for the quintic problem in dimension one \cite{MRlog}: $$\l(t)\sim\sqrt{\frac{T-t}{|\log|\log(T-t)|}}.$$

A second class of problems concerns anisotropic (NLS) problems like $$i\pa_tu+\pa_{xx}u-\pa_{yy}^2u+u|u|^{p-1}=0, \ \ (x,y)\in \Bbb R^2$$ which have triggered a lot of attention in particular regarding numerical simulations \cite{FP1,FP2,LNM} or the construction of infinite energy self similar solutions \cite{KN}, and where anisotropic blow up with very different behaviours in the $x,y$ directions is expected.


\subsection{The cylindrical blow up problem}


We propose in this paper a systematic program for the construction of parabolic anisotropic blow up bubbles by a dimensional reduction from $d$ to $d-1$. In order to set up the problem and for the sake of simplicity, we consider the {\it four} dimensional focusing semilinear heat equation
\be
\label{nlh}
\pa_tu=\Delta u+|u|^{p-1}u, \ \ x\in \Bbb R^4
\ee
in the energy super critical zone $$p>5.$$ Note that $p=5$ is the critical power for $\Bbb R^3$ which is $p=3$ for $\Bbb R^4$. We decompose $x\in\Bbb R^4$ as
$$x=(x',z)\in \Bbb R^3\times \Bbb R, \ \ r=|x'|$$ 
and consider functions $f$ on $\Bbb R^4$ which have cylindrical symmetry and are even with respect to $z$, i.e. 
$$f(x)=f(r,z), \  \ f(r,-z)=f(r,z).$$ We call this symmetry even cylindrical symmetry. Since for any rotation matrix $R$ of $\Bbb R^3$ the transformations $u(t,x',z)\to u(t, Rx', z)$ and $u(t,x',z)\to u(t, x', -z)$ map a solution to \eqref{nlh} onto another solution to \eqref{nlh}, uniqueness provided by the Cauchy theory ensures that the even cylindrical symmetry is propagated by the flow.\\

Since $p>5$, $p_{JL}=+\infty$ and the only known blow up bubbles correspond to type I blow up bubbles with either the ODE or a non trivial 4 dimensional self similar profile. A basic observation however is now that any {\it three} dimensional radially symmetric type I self similar solution $u(t,r)$ as constructed in \cite{CRS} is formally a solution to the {\it four} dimensional equation \eqref{nlh}, but this solution is constant in the $z$ direction. Hence it has infinite energy and its dynamical role among solutions to \eqref{nlh} is unclear. 

\subsection{Statement of the result}

Our main claim is that there exist a blow up scenario emerging from finite energy initial data which to leading order reproduces the self similar three dimensional blow up. Hence this solution is nearly constant along the $z$ axis in a boundary layer $|z|\leq z(t)$. We equivalently claim that the 3-dimensional self similar blow up is transversally stable modulo a finite number of instability directions.

\begin{theorem}[Finite codimensional transversal stability of self similar blow up]
\label{thmmain} 
Let $\Phi(r)$ solve \eqref{vnnevokenone}, \eqref{propertiesphi} and assume that the following non degeneracy condition is fulfilled:  let $(\l_j)_{-\ell_0\leq j\leq -1}$ be given by \eqref{spectralgapeigenvalueintro}, then 
\be
\label{hypspectral}
\forall j\in \{-\ell_0,\dots,-2\}, \ \ -\l_j\notin \Bbb N.
\ee
Then,  there exists a finite codimensional smooth manifold of initial data $u_0$ with even cylindrical symmetry and finite energy satisfying \eqref{eq:theinitialformoftheinitialdata} \eqref{eq:theinitialformoftheinitialdatabis} such that the corresponding solution to \eqref{nlh} blows up in finite time $T<+\infty$ with the following sharp description of the singularity. For $t$ close enough to $T$, the solution decomposes in self similar variables $$u(t,x)=\frac{1}{(T-t)^{\frac{2}{p-1}}}U(t,Y), \ \ Y=\frac{x}{\sqrt{T-t}}$$ as $$U(t,Y)=\frac{1}{(1+b(t)z^2)^{\frac 1{p-1}}}\Phi\left(\frac{r}{\sqrt{1+b(t)z^2}}\right)+v(t,Y), \ \ Y=(r,z)$$ with $$\lim_{t\to T}\|v(t,\cdot)\|_{L^{\infty}}=0,$$ and the free boundary moves at the speed 
\be
\label{movementboundary}
\frac{1}{\sqrt{b(t)}}=c^*(1+o(1))\sqrt{|\log (T-t)|}, \ \ c^*=c^*(\Phi)>0.
\ee
\end{theorem}

\noindent{\it Comments on the result}\\

\noindent{\it 1. Moving free boundary}. The main feature of Theorem \ref{thmmain} is to exhibit blow up solutions with strongly {\it anisotropic} blow up profiles. In particular the solution is almost constant in $z$ and equals the three dimensional self similar profile $\Phi(r)$ inside the boundary layer  $$|z(t)|\lesssim \sqrt{(T-t)|\log(T-t)|}$$ and the heart of the proof is to precisely compute the boundary. Note that the singularity still occurs at a point and not along the full $z$ axis. The free boundary is computed by constructing the reconnecting profile which generalizes the construction in \cite{BK,MZduke} for the ODE profile, and showing its stability. Note that in the companion paper \cite{CMRtypeII}, the transversal stability of type II blow up is proved and leads to a {\it completly different} behaviour of the free boundary.\\

\noindent{\it 2. On the spectral assumption \eqref{hypspectral}}. We expect the spectral assumption \eqref{hypspectral} to be generic. It would aslo typically be fulfilled for the minimizing self similar solution of the energy super critical heat flow (for which $\l_{-1}=-1$ is the bottom of the spectrum of $\L_r$).  In the setting of the construction of solutions by bifurcation \cite{BK, CRS}, this condition can be checked numerically, \cite{BK}. Let us stress that our analysis suggests that other integer eigenvalues generate new zeros of the full four dimensional linearized operator close to $\Phi(r)$, see Lemma \ref{lemmaspectralone}, and hence can generate new moving boundaries. Let us also stress that the speed of the moving boundary \eqref{movementboundary} is the {\it fundamental} mode, and that other speeds could be constructed corresponding to higher order excited modes.\\

\noindent{\it 3. More dimensional reductions}. More generally, one could address the following problem: consider the heat equation $$\pa_tu=\Delta u+u^p, \ \ x\in \Bbb R^{d_1}\times\Bbb R^{d_2} \ \ \mbox{with}\ \ \frac{d_1+2}{d_1-2}<p<p_{JL}(d_1)$$ with $p_{JL}$ given by \eqref{exponentpjl}, can one construct a self similar blow up dynamics emerging from finite energy initial which is to leading order constant in the direction $(x_{d_1+1},\dots,x_{d_1+d_2})$? \\
Theorem \ref{thmmain}  gives a positive answer for $d_2=1$, and we expect that it is the first step of an iteration argument. This would produce for a given nonlinearity finite energy self similar blow up solutions in arbitrarily large dimensions which is not known as of today. \\

\noindent{\it 4. $L^\infty$ bounds}. The main difficulty of the analysis is to control the perturbation in $L^{\infty}$ in order to deal with the nonlinear term. Here the computation of the free boundary and the construction of the {\it reconnecting} profiles, Lemma \ref{reconnectingprofiles}, is {\it essential}. Such estimates were derived for the ODE blow up problem in \cite{BK} using {\it explicit resolvent estimates} for the linearized flow near the constant self similar solution, and in \cite{MZgaffa} using general Liouville type classification theorem. These approaches are not obvious to implement here due to to the super critical nature of the problem, and the fact that there is no explicit formula for $\Phi$. We will overcome this using new elementary $W^{1,q}$ energy estimates, and a by product of our analysis is another self contained dynamical proof of the stability of the ODE type I blow up using purely energy estimates.\\

\subsection*{Notations}

We let $$Y=(y,z)\in \Bbb R^3\times \Bbb R, \ \ r=|y|$$ be the renormalized space variable. We let 
$$\Delta_r=\pa_r^2+\frac{2}{r}\pa_r, \ \ \Delta_Y=\Delta_r+\pa_z^2,$$ and the generator of scalings be $$\Lambda_r=\frac{2}{p-1}+r\pa_r, \ \ \Lambda_Y=\frac{2}{p-1}+Y\cdot\nabla.$$ We define the weights $$\rho_r=e^{-\frac{r^2}{4}}, \ \ \rho_Y=e^{-\frac{|Y|^2}{4}}, \ \ \rho_z=e^{-\frac{z^2}4}$$ with associated weighted norm $$\|u\|_{L^2_{\rho_r}}^2=\int_{\Bbb R^4}|u|^2\rho_r dY, \ \ \|u\|_{L^2_{\rho_Y}}^2=\int_{\Bbb R^4}|u(Y)|^2\rho_YdY.$$  We say a function $u(Y)$ has even cylindrical symmetry if $$u(Y)=u(r,z)=u(r,-z)$$ and denote $$L^{2,e}_{\rho_Y}$$ the associated Hilbert space. We let $\Phi(r)$ be a three dimensional self similar solution
\bea\label{eq:selfsimilarequationprofilPhi}
\Delta_r \Phi-\frac 12\Lambda_r \Phi+\Phi^p=0
\eea
satisfying \eqref{propertiesphi} as build in \cite{CRS}. We define for $m\in \Bbb N$ the $m$-th one dimensional Hermite polynomial 
\be
\label{hermitepolynomial}
P_m(z)=\sum_{k=0}^{[\frac m2]}\frac{m!}{k!(m-2k)!}(-1)^kz^{m-2k}
\ee
 which satisfy 
 $$\int_{\Bbb R} P_mP_{m'}\rho_zdz=\sqrt{\pi}2^{m+1} m!\delta_{mm'}.$$ 
 We let $$s_c=\frac32-\frac{2}{p-1}, \ \ S_c=2-\frac{2}{p-1},$$
where $s_c$ is the 3d critical exponent and $S_c$ is the 4d critical exponent. We let $$\la x\ra=\sqrt{1+|x|^2}.$$


\section{Approximate solution in the boundary layer}



\subsection{Reconnecting profiles}


Consider the renormalization $$u(t,x)=\frac{1}{\l(t)^{\frac{2}{p-1}}}U(s,Y), \ \ \frac{ds}{dt}=\frac{1}{\l^2}, \ \ Y=\frac{x}{\l(t)}$$ which maps \eqref{nlh} onto 
\be
\label{renormalizedflow}
\pa_sU=\Delta_YU+\lsl\Lambda_YU+U^p.
\ee 
For the self similar choice $$-\lsl=\frac 12,$$ an exact solution is given by $U(Y)=\Phi(r)$, but this solution does not decay along the $z$ direction. A better approximate solution decaying as $|Y|\to +\infty$ can be constructed by generalizing the approach in \cite{BK, MZduke}.

\begin{lemma}[Reconnecting profiles]
\label{reconnectingprofiles}
For all $b>0$, 
\be
\label{defphib}
\Phi_b(r,z)=\frac{1}{\mu_b(z)^{\frac2{p-1}}}\Phi\left(\frac{r}{\mu_b(z)}\right)\ \ \mbox{with}\ \ \mu_b(z)=\sqrt{1+bz^2}
\ee 
solves 
\be
\label{reconnectingequations}
\frac12z\pa_z\Phi_b=\Delta_r\Phi_b-\frac 12\Lambda_r\Phi_b+\Phi_b^p.
\ee
\end{lemma}

\begin{proof} 
On the one hand, we have
\bee
\frac{1}{2}z\pa_z\Phi_b &=& -\frac{1}{2}\frac{bz^2}{\mu_b^{\frac{2}{p-1}+2}}\Lambda_r\Phi\left(\frac{r}{\mu_b}\right),
\eee
and on the other hand, we have
\bee
&&\Delta_r\Phi_b-\frac 12\Lambda_r\Phi_b+\Phi_b^p = \frac{1}{\mu_b^{\frac{2}{p-1}+2}}\left(\Delta_r\Phi -\frac{\mu_b^2}{2}\Lambda_r\Phi+\Phi^p\right)\left(\frac{r}{\mu_b}\right)\\
&=& \frac{1}{\mu_b^{\frac{2}{p-1}+2}}\frac{1}{2}(1-\mu_b^2)\Lambda_r\Phi\left(\frac{r}{\mu_b}\right)= -\frac{1}{2}\frac{bz^2}{\mu_b^{\frac{2}{p-1}+2}}\Lambda_r\Phi\left(\frac{r}{\mu_b}\right),
\eee
where we used \eqref{eq:selfsimilarequationprofilPhi}. This concludes the proof of the lemma.
\end{proof}


\subsection{Diagonalization of the linearized operator close to $\Phi$}
\label{sectiondiag}

Let the $4$ dimensional linearized operator close to $\Phi$: $$\L_Y=-\Delta_Y +\frac 12\Lambda_Y-p\Phi^{p-1},$$ then $\L_Y$ is self adjoint on a domain $\mathcal D(\L_Y)\subset L^2_{\rho_Y}(\Bbb R^4)$ and with compact resolvent. Let the $3$ dimensional radial operator $$ \L_r=-\Delta_r+\frac 12\Lambda_r-p\Phi^{p-1}$$ which is self adjoint on a domain $\mathcal D(\L_r)\subset L^2(r^2\rho_rdr)$ with compact resolvent and spectrum determined in \cite{CRS}: 

\begin{lemma}[Spectrum for $\L_r$ in weighted spaces, \cite{CRS}]\label{lemma:spectrumoftheoperatormathcalLr} 
The spectrum of $\L_r$ with domain $\mathcal D(\L_r)\subset L^2(r^2\rho_rdr)$ is given by 
$$
\lambda_{-\ell_0}<\dots<\lambda_{-1}=-1<0<\l_{0}<\l_{1}<\dots
$$
for some integer $\ell_0\geq 1$ with 
\be
\label{spectralgapeigenvalue}
\l_{j}>0\textrm{ for all }j\geq 0 \textrm{ and }\lim_{j\to+\infty}\l_{j}=+\infty.
\ee
\item The eigenvalues $(\lambda_{j})_{-\ell_0\leq j\leq -1}$ are simple and associated to spherically symmetric eigenvectors 
\be
\label{formulapsiminustow}
\psi_{-j}(r), \ \ \|\psi_{-j}\|_{L^2(r^2\rho_rdr)}=1, \ \ \psi_{-1}=\frac{\Lambda_r \Phi}{\|\Lambda_r \Phi\|_{L^2(r^2\rho_rdr)}}.
\ee
Moreover, there holds the bound as $r=|y|\to +\infty$
\be
\label{estgrowth}
|\pa_k\psi_{j}(r)|\lesssim (1+r)^{-\frac{2}{p-1}-\lambda_{j}-k}, \ \ -\ell_0\leq j\leq -1, \ \ k\geq 0.
\ee
\end{lemma}

We may now diagonalize the full operator $\L_Y$ for function with cylindrical symmetry using a standard separation of variables claim and the tensorial structure of $\L_Y$.

\begin{lemma}[Spectrum for $\L_Y$ in weighted spaces with cylindrical symmetry] 
\label{lemmaspectralone}
The spectrum of $\L_Y$ restricted to functions of cylindrical symmetry with domain $\mathcal D(\L_Y)\subset L^2_{\rho_Y}(\Bbb R^4)$ is given by 
$$
\mu_{j,m}=\lambda_j+\frac{m}{2}, \ \ j\in [-\ell_0,+\infty), \ \ m\in \Bbb N
$$
with eigenfunction 
\be
\label{cnecneoneo}
\phi_{j,m}(Y)=\psi_{j}(r)P_m(z)
\ee where $P_m(z)$ is the $m$-th one dimensional  Hermite polynomial \eqref{hermitepolynomial} and $\psi_j$ denote the eigenvectors of $\L_r$. In particular, for $-\ell_0\leq j\leq -1$, let $m(j)$ be the smallest integer such that $$\frac{m(j)+1}{2}+\l_j>0,$$ then there holds the spectral gap estimate: $\forall \e\in H^1_{\rho_Y}$, 
$$
(\L_Y\e,\e)_{L^2_{\rho_Y}}\geq c\|\e\|_{H^1_{\rho_Y}}^2-\sum_{j=-\ell_0}^{-1}\sum_{m=0}^{m(j)}(\e,\phi_{j,m})^2_{L^2_{\rho_Y}}
$$
for some universal constant $c>0$.
\end{lemma}

\begin{remark}\label{remark:trivialmemberkernel} 
In particular $\mu_{-1,2}=-1+\frac{2}{2}=0$, and hence there is always a zero eigenmode. In view of \eqref{cnecneoneo} for $j=-1$ and $m=2$, formula \eqref{formulapsiminustow} for $\psi_{-1}$ and formula \eqref{hermitepolynomial} for $P_2$, the corresponding eigenvector is given by
$$(z^2-2)\Lambda \Phi.$$ 
\end{remark}

\begin{proof} This is a standard claim based on separation of variables. We compute $$\L_Y(\psi(r)P_m(z))=P_m(z)\left[-\Delta_r +\frac 12\Lambda_r-p\Phi^{p-1}\right]\psi(r)+\psi(r)\left[-\pa_{zz} +\frac 12z\pa_z\right]P_m(z)$$ and hence for an eigenfunction $\L_r\psi_j=\lambda_j\psi_j$:
\bee
\L_Y(\psi_j(r)P_m(z)) &=& \psi_j(r)\left[-\pa_{zz} +\frac 12z\pa_z+\lambda_j\right]P_m(z)\\
&=& \psi_j(r)\left[\frac{m}{2}+\lambda_j\right]P_m(z),
\eee
where we used the fact that the one dimensional harmonic oscillator $$-\pa^2_z+\frac 12z\pa_z$$ has spectrum $\frac{m}{2}$, $m\in \Bbb N$ on $L^2_{\rho_z}$ with eigenfunctions given by the $m$-th Hermite polynomial $P_m(z)$. It remains to observe that $\psi_j(r)P_m(z)$ is a dense family of the cylindrically symmetric functions of $L^2_{\rho_Y}(\Bbb R^4)$ from standard tensorial claims to conclude that it forms a Hilbertian basis of eigenvectors. The spectral gap estimate \eqref{spectralgapestimate} then follows by decomposition of the self adjoint operator $\L_Y$ in the Hibertian basis $\phi_{j,m}$.
\end{proof}

Under the additional assumption of even cylindrical symmetry and the fact that $P_{2m}$ is an even polynomial while $P_{2m+1}$ is an odd polynomial for all $m\in\mathbb{N}$ from \eqref{hermitepolynomial}, we obtain as a direct consequence of Lemma \ref{lemmaspectralone}:

\begin{lemma}[Spectrum for $\L_Y$ in weighted spaces with even cylindrical symmetry] 
\label{lemmaspectraltwo}
The spectrum of $\L_Y$ with domain $\mathcal D(\L_Y)\subset L^{2,e}_{\rho_Y}(\Bbb R^4)$ is given by 
\be
\label{eignevalurenvon}
\mu_{j,2M}=\lambda_j+M, \ \ j\in [-\ell_0,+\infty), \ \ M\in \Bbb N
\ee
with eigenfunction $$\phi_{j,2M}(Y)=\psi_{j}(y)P_{2M}(z)$$ where $P_m(z)$ is the $m$-th one dimensional  Hermite polynomial \eqref{hermitepolynomial}. In particular, for $-\ell_0\leq j\leq -1$, let $M(j)$ be the smallest integer such that $$M(j)+1+\l_j>0,$$ then there holds the spectral gap estimate: $\forall \e\in H^{1,e}_{\rho_Y}$,
\be
\label{spectralgapestimate}
(\L\e,\e)_{L^2_{\rho_Y}}\geq c\|\e\|_{H^1_{\rho_Y}}^2-\frac 1c\sum_{j=-\ell_0}^{-1}\sum_{M=0}^{M(j)}(\e,\phi_{j,2M})^2_{L^2_{\rho_Y}}
\ee
for some universal constant $c>0$.
\end{lemma}


\subsection{The high order approximate solution in the boundary layer}


Let us consider again the renormalized flow \eqref{renormalizedflow}. The choice $$\left(\lsl=-\frac 12, U(Y)=\Phi_b(Y),b(s)=b>0\right)$$  yields an $O(b)$ approximate solution in the boundary layer $|z|\lesssim\frac{1}{\sqrt{b}}$. We aim at improving this error and construct a high order approximate solution for $|z|\ll\frac{1}{\sqrt{b}}$, which will be the key to the control of the flow in $L^\infty$.\\

Let us indeed pick a smooth mapping $s\mapsto b(s)$ with $0<b(s)\ll1$ and look for a solution to \eqref{renormalizedflow} of the form 
$$U(s,Y)=\Phi_{b(s)}(Y)+v(s,Y)$$ 
which together with \eqref{reconnectingequations} yields:
\be
\label{vequation}
\pa_sv+\L_Y v=\pa^2_z\Phi_b-\pa_s\Phi_b+\left(\lsl+\frac 12\right)(\Lambda_Y \Phi_b+\Lambda_Y v)+F(v)
\ee
where
\bea\label{defnl}
F(v)=(\Phi_b+v)^p-\Phi_b^p-p\Phi_b^{p-1}v+p(\Phi_b^{p-1}-\Phi^{p-1})v.
\eea
We shall solve an approximate version of \eqref{vequation}. First let $$Z=\sqrt{b}z$$ and $$\Phi_b(r,z)=G(r,Z), \ \ G(r,Z)=\frac{1}{\mu(Z)^{\frac{2}{p-1}}}\Phi\left(\frac{r}{\mu(Z)}\right), \ \ \mu(Z)=\sqrt{1+Z^2}.$$ In order to construct an approximate solution, we anticipate the laws 
\be
\label{anticiaptionmod}
b_s=-bB(b), \ \ \lsl+\frac 12=M(b)
\ee and look for a solution of the form $$v_{b(s)}(s,r,z)=V_{b(s)}(r,Z)$$ so that $$\pa_sv=-B(b)\left[b\pa_b+\frac 1{2}Z\pa_Z\right]V, \ \  \pa_z^2v=b\pa_Z^2V, \ \ z\pa_z v=Z\pa_ZV$$ 
and \eqref{vequation} becomes:
\bee
\left(\L_r+\frac 12Z\pa_Z\right) V&=& b\pa^2_{Z}(G+V)+B(b)\left(\frac 1{2}Z\pa_ZG+\frac 12Z\pa_ZV+b\pa_bV\right)\\
&+& M(b)(\Lambda_r+Z\pa_Z)(G+ V)+ \widetilde{F}(V),
\eee 
where $\widetilde{F}(V)$ is defined by
\bee
\widetilde{F}(V)=(G+V)^p-G^p-pG^{p-1}V+p(G^{p-1}-\Phi^{p-1})V.
\eee
Given $0<\delta\ll1$, we let: 
\be
\label{defomegeadelta}
\Omega_{\delta}=\{ |Z|\leq \delta\},
\ee 
and construct an arbitrarily high order approximate solution in $\Omega_{\delta}$ using an elementary Hilbert expansion.

\begin{lemma}[High order approximate solution]
\label{lemmaapproximate}
Let $n\in \Bbb N^*$ such that $n\geq p$. Then for all $0<\delta<\delta(n)
\ll1 $ and $0<b<b(n)\ll 1$ small enough, there exist 
\bea
\label{definitionv}
V_b(r,Z)=\sum_{i=1}^n\sum_{j=0}^nb^iV_{i,j}(r)Z^{2j}, \ \ B(b)=\sum_{i=1}^nc_ib^i, \ \ M(b)=\sum_{i=1}^nd_ib^i
\eea 
where  
\be
\label{decayv}
|\pa_r^kV_{i,j}|\lesssim_{n,k} \frac{1}{\la r\ra^{\frac{2}{p-1}-\frac{1}{n}+k}}, \ \ k\in \Bbb N
\ee 
such that 
\be\label{orthogonality}
(V_{i,0},\Lambda_r\Phi)_{L^2_{\rho_r}}=(V_{i,1},\Lambda_r\Phi)_{L^2_{\rho_r}}=0,\,\,\,\, 1\leq i\leq n,
\ee 
and
\bea
\label{defeorror}
&&\Psi_b=\left(\L_r+\frac 12 Z\pa_Z\right) V_b-b\pa^2_{Z}(G+V_b)- F(V_b)\\
&\nonumber - &B(b)\left(\frac 1{2}Z\pa_Z(G+V_b)+b\pa_bV_b\right)-M(b)(\Lambda_r+Z\pa_Z)(G+ V_b)
\eea
satisfies  
\be
\label{esterrorinomega}
\forall Z\in \Omega_{\delta}, \ \ |\pa_r^j\pa_Z^k \Psi_b|\lesssim_{n} \frac{b^{n+1}+b|Z|^{2n+2-k}}{\la r\ra^{\frac{2}{p-1}-\frac{1}{n}+j}},\ \ \ 0\leq j+k\leq 2.
\ee
Moreover, there holds for the first terms:
\be
\label{degeneracy}
V_{1,0}=0
\ee
and 
\be
\label{calculloi}
c_1=2(2-s_c)+\frac{\|r\Lambda \Phi\|_{L^2_{\rho_Y}}^2}{2\|\Lambda \Phi\|_{L^2_{\rho_Y}}^2}, \ \ d_1=1.
\ee
\end{lemma}

\begin{remark} The law \eqref{anticiaptionmod}, \eqref{calculloi} written in the setting of the ODE type I blow up $\Phi=\left(\frac{1}{p-1}\right)^{\frac{1}{p-1}}$ yields the leading order $b$ law $$b_s+\frac{4p}{p-1}b^2=0$$ which is the frontier boundary computed in \cite{BK,MZduke}.
\end{remark}

\begin{proof} The proof follows by a brute force expansion.\\

\noindent{\bf step 1} Taylor expansion in $\Omega_{\delta}$. Recall the uniform bound 
\be
\label{uniformbound}
\frac{1}{\la r\ra^{\frac{2}{p-1}}}\lesssim \Phi(r)\lesssim \frac{1}{\la r\ra^{\frac{2}{p-1}}}
\ee
and 
\be
\label{boundphir}
\forall k\geq 1, \ \ |\Lambda^k\Phi|\lesssim_k \frac{1}{\la r\ra^{\frac{2}{p-1}+2}}.
\ee
Moreover, we compute 
$$\pa_ZG=-\frac{\mu'}{\mu}\frac{1}{\mu^{\frac{2}{p-1}}}\Lambda_r \Phi\left(\frac{r}{\mu(Z)}\right)$$
and a simple induction argument based on \eqref{boundphir} ensures for $k\geq 1$ the bound: 
\be
\label{cneknekonneo}
\forall |Z|\leq \delta, \ \ |\pa_Z^{2k}G(r,Z)|\lesssim_k \sum_{j=1}^{2k}|\Lambda^j_r\Phi|\lesssim \frac{1}{\la r\ra^{\frac{2}{p-1}+2}}.
\ee 
In particular, 
$$\left|\frac{\pa_Z^{2k}G}{\Phi}\right|\lesssim_k \frac{1}{\la r\ra^2}, \ \ k\geq 1.$$
We may therefore replace $G$ by its Taylor expansion at the origin \bee
&&G(r,Z)=G_n(r,0)+\frac{Z^{2n+2}}{(2n+1)!}\int_0^1(1-\tau)^{2n+1}\pa_Z^{2n+2}G(r,\tau Z)d\tau,\\
&& G_n(r,Z)=\sum_{k=0}^n\frac{\pa_Z^{2k}G(r,0)}{(2k)!}Z^{2k}
\eee
with for $|Z|\leq \delta$,
\be
\label{formulagn}
G_n(r,Z)=\Phi(r)\left[1+\sum_{k=1}^nZ^{2k}F_k(r)\right], \ \ F_k(r)\lesssim_k \frac 1{\la r\ra ^2}
\ee
and 
\bea\label{eq:estimateneededtoestimatePsi2}
|G-G_n|\lesssim_n \frac{Z^{2n+2}}{\la r\ra ^{\frac{2}{p-1}+2}}.
\eea

Next, let $\mu:\mathbb{R}^=\to \mathbb{R}^+$ a smooth cut-off function such that 
$$\mu=1\textrm{ on }0\leq s\leq 1\textrm{ and }\mu=0\textrm{ on }s\geq 2,$$
and let $\mu_b$ be defined by
\bee
\mu_b(r)=\mu(b^nr).
\eee
Note that for $|Z|\leq \de$ and $\de$ small enough, we have  
\bee
\frac{1}{b\la r\ra^{\frac{1}{n}}}\frac{|V_b|}{|G_n|} &\lesssim& \frac{1}{b\la r\ra^{\frac{1}{n}}}\frac{|V_b|}{\Phi}\lesssim \sum_{i=1}^n\sum_{j=0}^n b^{i-1}\de^j\la r\ra^{-\frac{1}{n}+\frac{2}{p-1}}|V_{i,j}|\lesssim b+\de 
\eee
where we anticipated on \eqref{degeneracy}. For $b$ and $\de$ small enough, we infer
\bea\label{eq:estimateneededtoestimatePsi2bis}
\frac{|V_b|}{|G_n|} \leq \frac{1}{2}\textrm{ on the support of }\mu_b.
\eea
We now Taylor expand the nonlinearity using 
$$(1+x)^p-1-px^{p-1}=\sum_{k=2}^{2n+1}a_kx^{k}+O(x^{2n+2}),\,\,\,\, |x|\leq \frac{1}{2}$$ 
which yields 
\be\label{eq:estimateneededtoestimatePsi2ter}
\mu_b(r)\Big((G_n+V_b)^{p}-G_n^{p}-pG_n^{p-1}V_b\Big)=\mu_b(r)\left(\sum_{k=2}^{2n+1}a_kV_b^kG_n^{p-k}+O(V_b^{2n+2}G_n^{p-(2n+2)})\right)\ee
Also, from \eqref{formulagn}: $\forall \alpha\in \Bbb Z$: 
\be\label{eq:estimateneededtoestimatePsi2quatre}
G_n^\alpha=\Phi(r)^\alpha\left(1+\sum_{j=1}^nZ^{2j}F_j(r)\right)^\alpha=\Phi(r)^\alpha\left[1+\sum_{j=1}^nZ^{2j}H_{\alpha,j}(r)+O(Z^{2n+2})\right]
\ee
with $$ |\partial_r^kH_{\alpha,j}(r)|\lesssim_k \frac1{\la r\ra^{2+k}}.$$ 
Thus, we decompose
\bea\label{eq:decompmositionofPsibn}
\Psi_b&=& \Psi^{(1)}_b + \Psi^{(2)}_b
\eea
where
\bea
\label{defeorrorbis}
\Psi^{(1)}_b&=&\left(\L_r+\frac 12 Z\pa_Z\right) V_b -  b\pa^2_{Z}(G_n+V_b)\\
\nonumber & - & \mu_b(r)\sum_{k=2}^{2n+1}a_k\left(\frac{V_b}{\Phi}\right)^k\Phi^{p}\left[1+\sum_{j=1}^nZ^{2j}H_{p-k,j}(r)\right]- p\Phi^{p-1}V_b\sum_{j=1}^nZ^{2j}H_{p-1,j}(r)\\
\nonumber &-&B(b)\left[\frac 1{2}Z\pa_Z(G_n+V_b)+b\pa_bV_b\right]-M(b)(\Lambda_r+Z\pa_Z)(G_n+ V_b).
\eea
and
\bea
\label{defeorrorremainderterms}
\Psi^{(2)}_b&=& -  b\pa^2_{Z}(G-G_n) - (1-\mu_b(r))\Big((G+V_b)^{p}-G^{p}-pG^{p-1}V_b\Big)\\
\nonumber & - & \mu_b(r)\left\{(G+V_b)^{p}-G^{p}-pG^{p-1}V_b - \sum_{k=2}^{2n+1}a_k\left(\frac{V_b}{\Phi}\right)^k\Phi^{p}\left[1+\sum_{j=1}^nZ^{2j}H_{p-k,j}(r)\right]\right\}\\
\nonumber&-& pG^{p-1}V_b+p\Phi^{p-1}V_b\left(1+\sum_{j=1}^nZ^{2j}H_{p-1,j}(r)\right)\\
\nonumber &-&B(b)\frac 1{2}Z\pa_Z(G-G_n) -M(b)(\Lambda_r+Z\pa_Z)(G-G_n).
\eea

\noindent{\bf step 2} Solving the approximate problem. We solve \eqref{defeorrorbis} up to an error of order $Z^{2n+2}$ or $b^{n+1}$ by looking for a solution of the form 
$$V_b(r,Z)=\sum_{i=1}^n\sum_{j=0}^nb^iV_{i,j}(r)Z^{2j}, \ \ B(b)=\sum_{i=1}^nc_ib^i, \ \ M(b)=\sum_{i=1}^nd_ib^i.$$ 
Since the polynomial dependance in both $b$ and $Z$ is preserved by the RHS of \eqref{defeorrorbis}, we sort the terms in $b^iZ^{2j}$ and obtain a hierarchy of equations of the following form for $1\leq i\leq n$, $0\leq j\leq n$ 
\bee
&&\left[\L_r+\frac 12 Z\pa_Z\right](V_{i,j}(r)Z^{2j})\\
&=& F_{i,j}(r)Z^{2j}+Z^{2j}\left|\begin{array}{ll} d_i\Lambda \Phi\ \ \mbox{for}\ \ j=0\\ \frac{c_i}{2(2j-1)!}\pa_Z^{2j}G(r,0)+\frac{d_i}{(2j)!}\left[\Lambda _r+2j\right]\pa_{Z}^{2j}G(r,0)\end{array}\right.
\eee
or equivalently:
\bea
\label{neiocnoneo}
&&\left[\L_r+j\right]V_{i,j}(r)\\
\nonumber &=& F_{i,j}(r)+\left|\begin{array}{ll} d_i\Lambda \Phi\ \ \mbox{for}\ \ j=0\\ \frac{c_i}{2(2j-1)!}\pa_Z^{2j}G(r,0)+\frac{d_i}{(2j)!}\left[\Lambda _r+2j\right]\pa_{Z}^{2j}G(r,0)\end{array}\right. 
\eea
where $F_{i,j}$ depends only on $V_{i',j'}$ with $i'\leq i$, $j'\leq j$ and $(i',j')\neq (i,j)$, and on $d_{i'}$ and $c_{i'}$ with $i'<i$. 
Moreover, a fundamental observation is that the decay \eqref{decayv} is preserved by the forcing term \eqref{defeorrorbis}, i.e. 
$$|\pa_r^kF_{i,j}(r)|\lesssim_n\frac{1}{ \la r\ra ^{\frac{2}{p-1}+k-\frac{1}{n}}},$$ 
where we used in particular the fact that for $2\leq k\leq 2n+1$, we have
\bee
\la r\ra^{\frac{2}{p-1}-\frac{1}{n}}\left(\frac{|V_b|}{\Phi}\right)^k\Phi^{p}  \lesssim \la r\ra^{\frac{2}{p-1}-\frac{1}{n}}\left(\la r\ra^{\frac{1}{n}}\right)^k\left(\frac{1}{\la r\ra^{\frac{2}{p-1}}}\right)^p\lesssim \la r\ra^{\frac{k-1}{n}-2} \lesssim 1.
\eee

In order to invert \eqref{neiocnoneo}, we will rely on the following lemma which is proved in Appendix \ref{sec:proofinversionoftheellipticsystemdefiningVij}. 
\begin{lemma}\label{lemma:inversionoftheellipticsystemdefiningVij}
Let $j\in\mathbb{N}$, and let $u_j(r)$ the solution to 
$$(\L_r+j)u=f_j\textrm{ and }(u_1, \Lambda_r\Phi)=0\textrm{ if }j=1.$$
Furthermore, assume that we have in the case $j=1$
$$(f_1,\Lambda_r\Phi)_{L^2_{\rho_r}}=0.$$ 
Then, for $\eta>0$ and $k\in\mathbb{N}$, $u$ satisfies the following bound 
$$ \sum_{l=0}^k\left\|\la r\ra^{\frac{2}{p-1}+l-\eta}\partial_r^lu_j\right\|_{L^{\infty}}\lesssim_{k,\eta} \sum_{l=0}^k\left\|\la r\ra^{\frac{2}{p-1}+l-\eta}\partial_r^lf_j\right\|_{L^{\infty}}.$$ 
\end{lemma}

We may now come back to \eqref{neiocnoneo}. We consider first the case $j=0$, then $j=1$, and finally $j\geq 2$. 
\begin{itemize}
\item We have for $j=0$ 
\bee
\L_rV_{i,0}(r) &=& F_{i,0}(r)+ d_i\Lambda \Phi
\eee
and hence, in view of Lemma \ref{lemma:inversionoftheellipticsystemdefiningVij}, the exists a unique $V_{i,0}$ which in view of the above estimate for $F_{i,j}$ satisfies
$$|\pa_r^kV_{i,0}(r)|\lesssim_n\frac{1}{ \la r\ra ^{\frac{2}{p-1}+k-\frac{1}{n}}}.$$
Furthermore, projecting on $\Lambda_r\Phi$ and using the fact that $\L_r(\Lambda_r\Phi)=-\Phi$, we have
\bee
-(V_{i,0}, \Lambda_r\Phi)_{L^2_{\rho_r}} &=& (F_{i,0}, \Lambda_r\Phi)_{L^2_{\rho_r}}+ d_i\|\Lambda \Phi\|_{_{L^2_{\rho_r}}}^2
\eee
and we choose $d_i$ to enforce
\bea\label{eq:orthogonalityVi0}
(V_{i,0}, \Lambda_r\Phi)_{L^2_{\rho_r}}=0.
\eea

\item Also, since $\pa_Z^{2}G(r,0)=-\Lambda_r\Phi(r)$, we have for $j=1$
\bee
\left[\L_r+1\right]V_{i,1}(r) &=& F_{i,1}(r) - \frac{c_i}{2}\Lambda_r\Phi-\frac{d_i}{2}(\Lambda _r+2)\Lambda_r\Phi.
\eee
We choose $c_i$ to enforce 
\bea\label{eq:enforcmentoforthogonalitythanksotci}
 (F_{i,1}, \Lambda_r\Phi)_{L^2_{\rho_r}} - \frac{c_i}{2}\|\Lambda \Phi\|_{_{L^2_{\rho_r}}}^2 -\frac{d_i}{2}\Big((\Lambda _r+2)\Lambda_r\Phi, \Lambda_r\Phi\Big)_{L^2_{\rho_r}} =0.
\eea
Thus, we may apply Lemma \ref{lemma:inversionoftheellipticsystemdefiningVij}, and hence the exists a unique $V_{i,1}$ such that
\bea\label{eq:orthogonalityVi1}
(V_{i,1}, \Lambda_r\Phi)_{L^2_{\rho_r}}=0,
\eea
and which in view of the above estimate for $F_{i,j}$ satisfies
$$|\pa_r^kV_{i,1}(r)|\lesssim_n\frac{1}{ \la r\ra ^{\frac{2}{p-1}+k-\frac{1}{n}}}.$$
Note that \eqref{orthogonality} follows from \eqref{eq:orthogonalityVi0} and  \eqref{eq:orthogonalityVi1}.

\item Finally, for $j\geq 2$, we may apply Lemma \ref{lemma:inversionoftheellipticsystemdefiningVij}, and hence the exists a unique $V_{i,j}$ which in view of the above estimate for $F_{i,j}$ satisfies
$$|\pa_r^kV_{i,j}(r)|\lesssim_n\frac{1}{ \la r\ra ^{\frac{2}{p-1}+k-\frac{1}{n}}}.$$ 
\end{itemize}

\noindent{\bf step 3} Proof of the error estimate. We are now in position to prove the error estimate \eqref{esterrorinomega}. As all terms of the type $b^iZ^{2j}$ for $1\leq i\leq n$ and $0\leq j\leq n$ in \eqref{defeorrorbis} vanish due to the choice of $V_{i,j}$, and in view of the estimates for $\Phi$, $G$, $H_{\alpha, j}$, as well as the estimates of step 2 above for $V_{i,j}$, we infer
\bee
\forall Z\in \Omega_{\delta}, \ \ |\pa_r^j\pa_Z^k \Psi^{(1)}_b|\lesssim_{n} \frac{b^{n+1}+b|Z|^{2n+2-k}}{\la r\ra^{\frac{2}{p-1}-\frac{1}{n}+j}}\,\,\,\,\forall j, k.
\eee
Also, we have for $Z\in  \Omega_{\delta}$ and $0\leq j+k\leq 2$
\bee
\left|\pa_r^j\pa_Z^k\Big((G+V_b)^{p}-G^{p}-pG^{p-1}V_b\Big)\right| \lesssim_{n} \frac{1}{\la r\ra^{\frac{2p}{p-1}-\frac{p}{n}+j}},
\eee
where we used the fact that $j+k\leq p$, since $p>5$ and $j+k\leq 2$, which ensures that the above expression does not contain negative powers of $G+V_b$. In view of the support of $1-\mu_b$, we deduce  for $Z\in  \Omega_{\delta}$ and $0\leq j+k\leq 2$
\bee
\left|\pa_r^j\pa_Z^k\left((1-\mu_b)\Big((G+V_b)^{p}-G^{p}-pG^{p-1}V_b\Big)\right)\right| \lesssim_{n} \frac{(b^n)^{2-\frac{p-1}{n}}}{\la r\ra^{\frac{2}{p-1}-\frac{1}{n}+j}}\lesssim_{n} \frac{b^{n+1}}{\la r\ra^{\frac{2}{p-1}-\frac{1}{n}+j}}
\eee
where we used the fact that $n\geq p$ in the last inequality. The other terms of $\Psi^{(2)}_b$ defined in \eqref{defeorrorremainderterms} are estimated using  \eqref{eq:estimateneededtoestimatePsi2} \eqref{eq:estimateneededtoestimatePsi2bis} \eqref{eq:estimateneededtoestimatePsi2ter} \eqref{eq:estimateneededtoestimatePsi2quatre} which leads to
 \bee
\forall Z\in \Omega_{\delta}, \ \ |\pa_r^j\pa_Z^k \Psi^{(2)}_b|\lesssim_{n} \frac{b^{n+1}+b|Z|^{2n+2-k}}{\la r\ra^{\frac{2}{p-1}-\frac{1}{n}+j}},\ \ \ 0\leq j+k\leq 2.
\eee
In view of the decomposition \eqref{eq:decompmositionofPsibn} for $\Psi_b$, we immediately infer from the estimates for $\Psi_b^{(1)}$ and $\Psi_b^{(2)}$ the error estimate \eqref{esterrorinomega} for $\Psi_b$.\\

\noindent{\bf step 4} Computation of $F_{1,0}$ and $F_{1,1}$. In view of the definition of $F_{i,j}$ in \eqref{neiocnoneo}, we have
\bee
F_{1,0}+F_{1,1}Z^2 = \partial_Z^2G(r,Z)+p\Phi^{p-1}V_{1,0}Z^2H_{p-1,1}(r)+O(Z^4). 
\eee
We compute the Taylor expansion 
\bea
\label{expansionG}
\pa^2_ZG&=&-\Lambda_r \Phi(r)+\frac 32 Z^2(2\Lambda_r \Phi+\Lambda_r^2\Phi)(r)+O(Z^4),
\eea
which yields
\bea\label{eq:computationofF10F11}
F_{1,0}=-\Lambda_r \Phi(r),\,\,\,\, F_{1,1}=\frac 32 (2\Lambda_r \Phi+\Lambda_r^2\Phi)(r)+p\Phi^{p-1}V_{1,0}H_{p-1,1}(r).
\eea

\noindent{\it Proof of \eqref{expansionG}}. Recall that we have
$$G(r,Z)=\frac{1}{\mu(Z)^{\frac{2}{p-1}}}\Phi\left(\frac{r}{\mu(Z)}\right), \ \ \mu(Z)=\sqrt{1+Z^2}.$$ Then 
\be
\label{cnekoneoneo}
\pa_ZG=-\frac{\mu'}{\mu}\frac{1}{\mu^{\frac{2}{p-1}}}\Lambda_r \Phi\left(\frac{r}{\mu(Z)}\right).
\ee 
We further compute:
\bea
\label{formulag}
 \pa^2_ZG & = & \frac{1}{(1+Z^2)^{2}\mu^{\frac{2}{p-1}}}\left[(Z^2-1)\Lambda_r \Phi+Z^2\Lambda_r^2\Phi\right]\left(\frac{r}{\mu}\right).
\eea
We now Taylor expand at $Z=0$ and obtain in particular using the uniform bound on $\Lambda_r^i\Phi(r)$, $i=1,2,3$:
\be
\label{claculinterne}
\frac{1}{\mu^{\frac2{p-1}}}\Lambda_r\Phi\left(\frac{r}{\mu}\right)=\Lambda_r\Phi (r)-\frac{Z^2}2\Lambda_r^2\Phi(r)+O(Z^4)
\ee
which yields the Taylor expansion at the origin:
$$
\pa^2_ZG=-\Lambda_r \Phi(r)+\frac 32 Z^2(2\Lambda_r \Phi+\Lambda_r^2\Phi)(r)+O(Z^4).
$$
This concludes the proof of \eqref{expansionG}. \\

\noindent{\bf step 5} Computation of $V_{1,0}$, $d_1$ and $c_1$. From \eqref{neiocnoneo} for $j=0$, we have
\bee
\L_rV_{1,0}(r) &=& F_{1,0}(r)+ d_1\Lambda_r \Phi
\eee
which together with \eqref{eq:computationofF10F11} yields
\bee
\L_rV_{1,0}(r) &=& (d_1-1)\Lambda_r \Phi.
\eee
Since we choose $d_1$ to enforce the orthogonality \eqref{eq:orthogonalityVi0}, we immediately deduce 
$$V_{1,0}=0,\,\,\,\, d_1=1,$$
which proves in particular \eqref{degeneracy}.

Next, recall from \eqref{eq:enforcmentoforthogonalitythanksotci} that we choose $c_1$ to enforce 
\bee
 (F_{1,1}, \Lambda_r\Phi)_{L^2_{\rho_r}} - \frac{c_1}{2}\|\Lambda_r \Phi\|_{_{L^2_{\rho_r}}}^2 -\frac{d_1}{2}\Big((\Lambda_r+2)\Lambda_r\Phi, \Lambda_r\Phi\Big)_{L^2_{\rho_r}} =0
\eee
which together with \eqref{eq:computationofF10F11} and the computation of $V_{1,0}$ and $d_1$ above yields 
\bee
c_1 &=& 4+\frac{2(\Lambda^2_r\Phi, \Lambda_r\Phi)_{L^2_{\rho_r}}}{\|\Lambda_r \Phi\|_{_{L^2_{\rho_r}}}^2}\\
&=& 2(2-s_c)+\frac{\|r\Lambda_r \Phi\|_{_{L^2_{\rho_r}}}^2}{2\|\Lambda_r \Phi\|_{_{L^2_{\rho_r}}}^2},
\eee
where we used in the last inequality the following computation
\be
\label{integrationbyparts}
(\Lambda_rf,f)_{L^2_{\rho_r}}=-s_c\|f\|_{L^2_{\rho_r}}^2+\frac 14\|rf\|_{L^2_{\rho_r}}^2, \ \ s_c=\frac 32-\frac{2}{p-1}.
\ee
This finishes the proof of \eqref{calculloi} and hence of Lemma \ref{lemmaapproximate}.
\end{proof}

\begin{lemma}[High order localized approximate solution]
\label{lemmaapproximatelocalized}
Let $n\in \Bbb N^*$ such that $n\geq p$. For $0<\delta<\delta(n)
\ll1 $ and $0<b<b(n)\ll 1$ small enough, let $(V_b,B(b),M(b))$ be the approximate solution given by Lemma \ref{lemmaapproximate}. Let an even cut off function 
$$\chi_\delta(z)=\chi\left(\frac{Z}{\delta}\right), \ \ \chi(\sigma)=\left|\begin{array}{ll} 1 \ \ \mbox{for}\ \ |\sigma|\leq 1,\\ 0 \ \ \mbox{for}\ \ |\sigma|\geq 2\end{array}\right.,$$ 
and let 
$$\Phit_b=\Phi_b+\tilde{v}_b\textrm{ where }\tilde{v}_b=\chi_\delta v_b\textrm{ and }v_b(z)=V_b(Z).$$
Then, $\Phit_b$ satisfies  
 \be\label{eqphibtilde}
-bB(b)\pa_{b}\Phit_b+\left(\frac{1}{2}-M(b)\right)\Lambda_Y \Phit_b-\Delta \Phit_b+\Phit_b^p=\Psit_b
\ee
where
\be\label{eq:decompoistionPsitbinPsitb0andG}
\Psit_b = b(\chi_\delta-1)\pr_Z^2G+B(b)(\chi_\delta-1)Z\partial_ZG+\Psit_b^{(0)}
\ee
and where $\Psit_b^{(0)}$ is estimated by
\be\label{esterrorinomegabis}
|\pa_r^j\pa_Z^k \Psit_{b}^{(0)}|\lesssim_{\delta} \frac{b^{n+1}+b|Z|^{2n+2-k}}{\la r\ra^{\frac 2{p-1}-\frac{1}{n}}}{\bf 1}_{|Z|\leq 2\delta},\ \ \ 0\leq j+k\leq 2.
\ee
Furthermore,  $\Phit_b$ satisfies also 
\be\label{deriveebzeroprofil}
(\Phit_b)_{|b=0}=\Phi, \ \ \frac{\pa \Phit_b}{\pa b}_{|b=0}=-\frac 12(P_2+2P_0)(z)\Lambda_r\Phi.
\ee
\end{lemma}

\begin{proof}
Since $v_b(z)=V_b(Z)$, and in view of the equation \eqref{defeorror} satisfied by $V_b$, we infer
$$
\mathcal L_Yv_b=\pr_z^2\Phi_b +bB(b)(\pr_b\Phi_b+\pa_bv_b)+M(b)(\Lambda_Y \Phi_b+\Lambda_Y v_b)+F(v_b)+\Psi_b
$$
with $\Psi_b$ satisfying \eqref{esterrorinomega}. Since $\tilde{v}_b=\chi_\delta v_b$, we infer 
\be
\label{eqvtilde}
\mathcal L_Y\tilde{v}_b=\chi_\delta\pr_z^2\Phi_b+bB(b)(\chi_\delta\pr_b\Phi_b+\pa_b\tilde{v}_b)+M(b)(\Lambda_Y \Phi_b+\Lambda_Y \tilde{v}_b)+F(\tilde{v}_b)+\Psit_b^{(0)}
\ee
with
\bee
\Psit_b^{(0)} &=& \chi\left(\frac{Z}{\delta}\right)\Psi_b + \left[\frac{1}{2}Z\chi'-\frac{b}{\de^2}\chi'' -\frac{B(b)}{\de}\chi'-M(b)Z\chi'\right]\left(\frac{Z}{\delta}\right)V_b \\
&-& \frac{2b}{\de}\chi'\left(\frac{Z}{\delta}\right)\partial_ZV_b+\left(1-\chi\left(\frac{Z}{\delta}\right)\right)G^p+\chi\left(\frac{Z}{\delta}\right)(G+V_b)^p - \left(G+\chi\left(\frac{Z}{\delta}\right)V_b\right)^p.
\eee
In view of the estimate \eqref{esterrorinomega} for $\Psi_b$, the properties of the support of $\chi$ and the estimates for $G$ and $V_b$, we immediately infer for $j+k\leq 2$
\bee
\nonumber\left|\pa_r^j\pa_Z^k \Psit_{b}^{(0)}\right|&\lesssim_{\delta}& \frac{b^{n+1}+b|Z|^{2n+2-k}}{\la r\ra^{\frac 2{p-1}-\frac{1}{n}}}{\bf 1}_{|Z|\leq 2\delta}+\frac{b}{\la r\ra^{\frac 2{p-1}-\frac{1}{n}}}{\bf 1}_{\delta\leq |Z|\leq 2\delta}\\
&\lesssim_{\delta}& \frac{b^{n+1}+b|Z|^{2n+2-k}}{\la r\ra^{\frac 2{p-1}-\frac{1}{n}}}{\bf 1}_{|Z|\leq 2\delta}
\eee
which is \eqref{esterrorinomegabis}.

Next, since $\Phit_b=\Phi_b+\tilde{v}_b$ and in view of the definition of $\Psit_b$, we have
\bee
\Psit_b &=& -bB(b)\pa_{b}\Phit_b -\Delta_Y\Phit_b+\left(\frac{1}{2}-M(b)\right)\Lambda_Y \Phit_b-\Phit_b^p\\
&=& -\Delta_Y\Phi_b + \frac{1}{2}\Lambda_Y\Phi_b -\Phi_b^p\\
&& +\L_Y\tilde{v}_b -bB(b)(\pr_b\Phi_b+\pa_b\tilde{v}_b) -M(b)\Lambda_Y(\Phi_b+\tilde{v}_b)-(\Phi_b+\tilde{v}_b)^p+\Phi_b^p+p\Phi^{p-1}\tilde{v}_b\\
&=&-\partial^2_z\Phi_b +\L_Y\tilde{v}_b -bB(b)(\pr_b\Phi_b+\pa_b\tilde{v}_b) -M(b)\Lambda_Y(\Phi_b+\tilde{v}_b)-F(\tilde{v}_b)
\eee
where we have used the equation \eqref{reconnectingequations} for $\Phi_b$ and the definition of $F$ in the last equality. Plugging \eqref{eqvtilde}, we infer
\bee
\Psit_b &=& (\chi_\delta-1)\pr_z^2\Phi_b+bB(b)(\chi_\delta-1)\pr_b\Phi_b+\Psit_b^{(0)}\\
&=& b(\chi_\delta-1)\pr_Z^2G+B(b)(\chi_\delta-1)Z\partial_ZG+\Psit_b^{(0)}
\eee
wich is \eqref{eq:decompoistionPsitbinPsitb0andG}.

Finally, we prove \eqref{deriveebzeroprofil}. We compute from \eqref{defphib}: 
$${\Phi_b}_{|_{b=0}}=\Phi$$
and
$$\frac{\pa \Phi_b}{\pa b}=-\frac{\pa_b\mu_b}{\mu_b}\frac{1}{\mu_b^{\frac 2{p-1}}}\Lambda_r\Phi\left(\frac{r}{\mu_b}\right)=-\frac{z^2}{2\mu_b^2}\frac{1}{\mu_b^{\frac 2{p-1}}}\Lambda_r\Phi\left(\frac{r}{\mu_b}\right),\,\,\,\, \mu_b=\sqrt{1+bz^2},$$ 
and hence $$\frac{\pa \Phi_b}{\pa b}_{|b=0}=-\frac{z^2}{2}\Lambda_r \Phi=-\frac 12(P_2+2P_0)(z)\Lambda_r\Phi$$ where we used from \eqref{hermitepolynomial}: $$P_2(z)=z^2-2, \ \ P_0(z)=1.$$ 
Moreover, we have
$$v_b(z)=V_b(z),\,\,\,\, \pr_bv_b(z)=\pr_bV_b(Z)+\frac{1}{2b}Z\partial_ZV_b(Z)$$
which together with \eqref{definitionv}, \eqref{degeneracy} yields\footnote{Recall that $Z=\sqrt{b}z$.}
$$(\tilde{v}_{b})_{|b=0}=(\pa_b\tilde{v}_b)_{|b=0}=0.$$ 
Hence, we infer
\bee
(\Phit_b)_{|b=0}=\Phi, \ \ \frac{\pa \Phit_b}{\pa b}_{|b=0}=-\frac 12(P_2+2P_0)(z)\Lambda_r\Phi
\eee
which is  \eqref{deriveebzeroprofil}. This concludes the proof of the lemma.
\end{proof}


\section{The bootstrap argument}



\subsection{Setting of the bootstrap}


We set up in this section the bootstrap analysis of the flow for a suitable set of finite energy initial data. The solution will be decomposed in a suitable geometrical way using by now standards arguments, see \cite{mamerle, meraphannals}.\\

\noindent { \em Geometrical decomposition of the flow}. We start by showing the existence of the suitable decomposition. 

\begin{lemma}[Geometrical decomposition] \label{lemma:decompositionofinitialdata}
There exists $\hat{b}>0$ and $\kappa>0$ small enough such if 
$$0<\underline{b}\leq\hat{b}\textrm{ and }\|w\|_{L^\infty}\leq \kappa,$$
and
$$u=\Phit_{\underline{b}}+w,$$ 
then $u$ has a unique decomposition
$$
u=\frac{1}{\mu^{\frac{2}{p-1}}}\left(\Phit_{b}+\sum_{j=-\ell_0}^{-2}\sum_{M=0}^{M(j)}a_{j,M}\phi_{j,2M}+\varepsilon\right)\left(\frac{x}{\mu} \right),
$$
where $\varepsilon$ satisfies the orthogonality conditions
$$
(\e,\phi_{j,2M})_{L^2_{\rho_Y}}=0, \ \ -\ell_0\leq j\leq -1, \ \ 0\leq M\leq M(j),
$$
and with
\be\label{eq:bound continuite decomposition}
|\mu-1|+|b-\underline{b}|+\sum_{j=-\ell_0}^{-2}\sum_{M=0}^{M(j)} |a_{j,M}| \lesssim \|w\|_{L^\infty}.
\ee
Furthermore, for $K$ such that
\bea\label{eq:necessaryconditionK}
K\geq 1+\max_{-\ell_0\leq j\leq -1}M(j),
\eea
and $q$ such that
\bea\label{eq:necessaryconditionq}
q>1\,\,\,\,\textrm{ and }\,\,\,\,\frac{q+1}{p-1}>2,
\eea
we have 
\bea\label{eq:bound continuite decompositionbis}
\nonumber&&\| \e \|_{H^2_{\rho_Y}}+\|\nabla\e\|_{L^{2q+2}_{\rho_Y}} +\left(\int\frac{\ep^2}{1+z^{2K}}\rho_rdY\right)^{\frac{1}{2}}+\left(\int\frac{|\nabla\ep|^{2q+2}}{1+z^{2K}}\rho_rdY\right)^{\frac{1}{2q+2}}+\|v\|_{W^{1,2q+2}}\\
&\lesssim& \underline{b}^{-\frac{5}{4}}(\|w\|_{H^2}+\|w\|_{W^{1,2q+2}})
\eea
where
$$v=\sum_{j=-\ell_0}^{-2}\sum_{M=0}^{M(j)}a_{j,M}\phi_{j,2M}+\varepsilon.$$
\end{lemma}

\begin{proof}
It is a classical consequence of the implicit function theorem. \\

\noindent\textbf{step 1} Existence of the decomposition of $U$ and proof of \eqref{eq:bound continuite decomposition}. We introduce the smooth maps
$$F\Big(w, \mu, b, (a_{j,M})_{-\ell_0\leq j\leq -2, \,0\leq M\leq M(j)}\Big)=\mu^{\frac{2}{p-1}}\left(\Phit_{\underline{b}}+w\right)(\mu x) -\Phit_b - \sum_{j=-\ell_0}^{-2}\sum_{M=0}^{M(j)}a_{j,M}\phi_{j,2M}$$
and
$$G = \Big((F, \phi_{j,M})_{L^2_{\rho_Y}},\, -\ell_0\leq j\leq -1, \,0\leq M\leq M(j)\Big).$$ 
We immediately check that $G(0, 1, \underline{b},\ldots, 0)=0$. Also, from \eqref{deriveebzeroprofil}, \eqref{formulapsiminustow} and Lemma \ref{lemmaspectraltwo}, we have 
$$\left(\Lambda\Phi_r,\phi_{j,2M}\right)_{L^2_{\rho_Y}}=\left(\frac{\pa \Phit_b}{\pa b}_{|b=0},\phi_{j,2M}\right)_{L^2_{\rho_Y}}=0, \ \ -\ell_0\leq j\leq -2, \ \ 0\leq M\leq M(j),$$
and hence, we deduce that
$$\frac{\pr G}{\pr ( \mu, b, (a_{j,M})_{-\ell_0\leq j\leq -2, \,0\leq M\leq M(j)})}_{|_{(0, 1, 0, \ldots, 0)}}
=\left(\begin{array}{cc}A & 0\\
 0 & I
 \end{array}\right)$$
where $I$ is the $N$ by $N$ identity matrix with the integer $N$ is given by
$$N=\sum_{j=-\ell_0}^{-1}(1+M(j))$$
and where $A$ is the following 2 by 2 matrix
$$\left(\begin{array}{cc} (\frac{\pa \Phit_b}{\pa b}_{|b=0},\phi_{-1,0})_{L^2_{\rho_Y}} & (\Lambda_r \Phi,\phi_{-1,0})_{L^2_{\rho_Y}}\\  (\frac{\pa \Phit_b}{\pa b}_{|b=0},\phi_{-1,2})_{L^2_{\rho_Y}} & (\Lambda_r \Phi,\phi_{-1,2})_{L^2_{\rho_Y}}\end{array}\right).$$
Since we have 
\bee
|A| &=&\frac{1}{\|P_0\|_{L^2_{\rho_z}}\|P_2\|_{L^2_{\rho_z}}\|\Lambda_r\Phi\|_{L^2_{\rho_r}}^2}\left|\begin{array}{cc} \left(-\frac 12(P_2+2P_0)(z)\Lambda_r\Phi, P_0\Lambda_r\Phi\right)_{L^2_{\rho_Y}} & (\Lambda_r \Phi, P_0\Lambda_r\Phi)_{L^2_{\rho_Y}}\\  \left(-\frac 12(P_2+2P_0)(z)\Lambda_r\Phi,P_2\Lambda_r\Phi\right)_{L^2_{\rho_Y}} &0\end{array}\right|\\
&=&\frac 12\|P_0\|_{L^2_{\rho_z}}\|P_2\|_{L^2_{\rho_z}}\|\Lambda_r\Phi\|_{L^2_{\rho_r}}^2\neq 0,
\eee
we deduce that 
$$\frac{\pr G}{\pr ( \mu, b, (a_{j,M})_{-\ell_0\leq j\leq -2, \,0\leq M\leq M(j)})}_{|_{(0, 1, 0, \ldots, 0)}}$$
is invertible. Since $0<\underline{b}\leq\hat{b}\ll 1$, we infer by continuity and the fact that the set of invertible matrices is open that 
$$\frac{\pr G}{\pr ( \mu, b, (a_{j,M})_{-\ell_0\leq j\leq -2, \,0\leq M\leq M(j)})}_{|_{(0, 1, \underline{b},0, \ldots, 0)}}$$
is invertible. In view of the implicit function theorem, for $\kappa>0$ small enough, for any 
$$
\| w \|_{L^{\infty}}\leq \kappa
$$
there exists $(\mu , b, (a_{j,M})_{-\ell_0\leq j\leq -2, \,0\leq M\leq M(j)})$ and
$$\ep =F\Big(w, \mu, b, (a_{j,M})_{-\ell_0\leq j\leq -2, \,0\leq M\leq M(j)}\Big)$$
such that
$$u =\Phit_{\underline{b}}+w= \frac{1}{\mu ^{\frac{2}{p-1}}}\left(\Phit_b+\sum_{j=-\ell_0}^{-2}\sum_{M=0}^{M(j)}a_{j,M}\phi_{j,2M}+\ep \right)\left(\frac{x}{\mu}\right),$$
$$
(\e,\phi_{j,2M})_{L^2_{\rho_Y}}=0, \ \ -\ell_0\leq j\leq -1, \ \ 0\leq M\leq M(j),
$$
and the estimate \eqref{eq:bound continuite decomposition} holds true for the parameters, i.e.
\bee
|\mu-1|+|b-\underline{b}|+\sum_{j=-\ell_0}^{-2}\sum_{M=0}^{M(j)} |a_{j,M}| \lesssim \|w\|_{L^\infty}.
\eee

\noindent\textbf{step 2} Proof of \eqref{eq:bound continuite decompositionbis}. Recall that we have defined $\ep$ as
$$\ep=\mu^{\frac{2}{p-1}}\left(\Phit_{\underline{b}}+w\right)(\mu x) -\Phit_b - \sum_{j=-\ell_0}^{-2}\sum_{M=0}^{M(j)}a_{j,M}\phi_{j,2M}.$$
We infer
\bee
\ep &=& \tilde{\ep}+\mu^{\frac{2}{p-1}}w(\mu Y)  - \sum_{j=-\ell_0}^{-2}\sum_{M=0}^{M(j)}a_{j,M}\phi_{j,2M}.
\eee
where we have introduced the notation 
\bee
\tilde{\ep} &=& (\mu-1)\int_0^1(1+\sigma(\mu-1))^{\frac{2}{p-1}}\Lambda_Y\Phit_{\tilde{b}}((1+\sigma(\mu-1))Y)d\sigma \\
&& + (b-\underline{b})\int_0^1\partial_b\Phit_{\underline{b}+\sigma(b-\underline{b})}(Y)d\sigma.
\eee
We estimate
\bee
&&\| \e \|_{H^2_{\rho_Y}}+\|\nabla\e\|_{L^{2q+2}_{\rho_Y}} +\left(\int\frac{\ep^2}{1+z^{2K}}\rho_rdY\right)^{\frac{1}{2}}+\left(\int\frac{|\nabla\ep|^{2q+2}}{1+z^{2K}}\rho_rdY\right)^{\frac{1}{2q+2}}+\|v\|_{W^{1,2q+2}}\\
&\lesssim& \| \tilde{\ep} \|_{H^2_{\rho_Y}}+\|\nabla\tilde{\ep}\|_{L^{2q+2}_{\rho_Y}} +\left(\int\frac{\tilde{\ep}^2}{1+z^{2K}}\rho_rdY\right)^{\frac{1}{2}}+\left(\int\frac{|\nabla\tilde{\ep}|^{2q+2}}{1+z^{2K}}\rho_rdY\right)^{\frac{1}{2q+2}}+\|\tilde{\ep}\|_{W^{1,2q+2}}\\
&& +\|w\|_{H^2}+\|w\|_{W^{1,2q+2}}+ \sum_{j=-\ell_0}^{-2}\sum_{M=0}^{M(j)}|a_{j,M}|,
\eee
where we used the fact that for $-\ell_0\leq j\leq -2$ and $0\leq M\leq M(j)$, we have
\bee
&&\left(\int\frac{\phi_{j,2M}^2}{1+z^{2K}}\rho_rdY\right)^{\frac{1}{2}}+\left(\int\frac{|\nabla\phi_{j,2M}|^{2q+2}}{1+z^{2K}}\rho_rdY\right)^{\frac{1}{2q+2}}\\
 &\lesssim& \left(\int\frac{P_{2M}^2}{1+z^{2K}}dz\right)^{\frac{1}{2}}+\left(\int\frac{(P_{2M}')^{2q+2}}{1+z^{2K}}dz\right)^{\frac{1}{2q+2}}\lesssim  1
\eee
in view of the choice
$$K\geq 1+\max_{-\ell_0\leq j\leq -1}M(j).$$
Together with the estimate for $a_{j,M}$ derived in step 1, we infer
\bee
&&\| \e \|_{H^2_{\rho_Y}}+\|\nabla\e\|_{L^{2q+2}_{\rho_Y}} +\left(\int\frac{\ep^2}{1+z^{2K}}\rho_rdY\right)^{\frac{1}{2}}+\left(\int\frac{|\nabla\ep|^{2q+2}}{1+z^{2K}}\rho_rdY\right)^{\frac{1}{2q+2}}+\|v\|_{W^{1,2q+2}}\\
&\lesssim& \| \tilde{\ep} \|_{H^2_{\rho_Y}}+\|\nabla\tilde{\ep}\|_{L^{2q+2}_{\rho_Y}} +\left(\int\frac{\tilde{\ep}^2}{1+z^{2K}}\rho_rdY\right)^{\frac{1}{2}}+\left(\int\frac{|\nabla\tilde{\ep}|^{2q+2}}{1+z^{2K}}\rho_rdY\right)^{\frac{1}{2q+2}}+\|\tilde{\ep}\|_{W^{1,2q+2}}\\
&& +\|w\|_{H^2}+\|w\|_{W^{1,2q+2}}
\eee
where we used the fact that $q>1$ and the Sobolev embedding in $\Bbb R^4$ in the last inequality. 

We still need to estimate $\tilde{\ep}$. We have
\bee
\Lambda_Y\Phi_b(Y)= \frac{1-Z^2}{1+Z^2}\frac{1}{\mu^{\frac{2}{p-1}}}\Lambda_r\Phi\left(\frac{r}{\mu}\right),\ \pr_b\Phi_b=-\frac{1}{b}\frac{Z^2}{1+Z^2}\frac{1}{\mu^{\frac{2}{p-1}}}\Lambda_r\Phi\left(\frac{r}{\mu}\right),
\eee
which together with the decay of $\Phi$, the fact that $\Phit_b=\Phi_b+\tilde{v}_b$ and the estimates for $\tilde{v}_b$ yields
\bee
|\pr^j_r\pr^k_Z\Lambda_Y\Phit_b(Y)|\lesssim \frac{1}{(\la r\ra+|Z|)^{\frac{2}{p-1}-\frac{1}{n}}},\,\,\,\, |\pr^j_r\pr^k_Z\pr_b\Phit_b(Y)|\lesssim \frac{1}{b(\la r\ra+|Z|)^{\frac{2}{p-1}-\frac{1}{n}}}
\eee
In view of the definition of $\tilde{\ep}$, we infer
\bee
&& \| \tilde{\ep} \|_{H^2_{\rho_Y}}+\|\nabla\tilde{\ep}\|_{L^{2q+2}_{\rho_Y}} +\left(\int\frac{\tilde{\ep}^2}{1+z^{2K}}\rho_rdY\right)^{\frac{1}{2}}+\left(\int\frac{|\nabla\tilde{\ep}|^{2q+2}}{1+z^{2K}}\rho_rdY\right)^{\frac{1}{2q+2}}+\|\tilde{\ep}\|_{W^{1,2q+2}}\\
 &\lesssim& \underline{b}^{-\frac{1}{4}}|\mu-1|+\underline{b}^{-\frac{5}{4}}|b-\underline{b}|
 \eee
 where we have used for the last term the fact that in view of $n\geq p$ and \eqref{eq:necessaryconditionq}, we have 
 $$(2q+2)\left(\frac{2}{p-1}-\frac{1}{n}\right)>4,$$
 so that we have in view of $Z=\sqrt{b}z$
 $$\left\|\frac{1}{(\la r\ra+|Z|)^{\frac{2}{p-1}-\frac{1}{n}}}\right\|_{L^{2q+2}}\lesssim \underline{b}^{-\frac{1}{4q+4}}.$$
 Together with the estimate for the parameters $b$ and $\mu$, we infer
\bee
&& \| \tilde{\ep} \|_{H^2_{\rho_Y}}+\|\nabla\tilde{\ep}\|_{L^{2q+2}_{\rho_Y}} +\left(\int\frac{\tilde{\ep}^2}{1+z^{2K}}\rho_rdY\right)^{\frac{1}{2}}+\left(\int\frac{|\nabla\tilde{\ep}|^{2q+2}}{1+z^{2K}}\rho_rdY\right)^{\frac{1}{2q+2}}+\|\tilde{\ep}\|_{W^{1,2q+2}}\\
 &\lesssim& \underline{b}^{-\frac{5}{4}}\|w\|_{L^\infty}.
 \eee 
 Coming back to $\ep$, we deduce
 \bee
&&\| \e \|_{H^2_{\rho_Y}}+\|\nabla\e\|_{L^{2q+2}_{\rho_Y}} +\left(\int\frac{\ep^2}{1+z^{2K}}\rho_rdY\right)^{\frac{1}{2}}+\left(\int\frac{|\nabla\ep|^{2q+2}}{1+z^{2K}}\rho_rdY\right)^{\frac{1}{2q+2}}+\|v\|_{W^{1,2q+2}}\\
&\lesssim& \underline{b}^{-\frac{5}{4}}(\|w\|_{H^2}+\|w\|_{W^{1,2q+2}})
\eee
 which is \eqref{eq:bound continuite decompositionbis}. This concludes the proof of the lemma.
\end{proof}

\noindent { \em Description of the initial data}. We now pick an initial data close to $\Phit_b$ up to scaling, where $\Phit_b$ has been constructed in Lemma \ref{lemmaapproximatelocalized}, and assume in the coordinate of the above geometrical decomposition
\bea\label{geomdeopcopmtiinit}
u_0=\frac{1}{\l_0^{\frac{2}{p-1}}}\Big(\Phit_{b_0}+v_0\Big)\left(\frac{x}{\l_0}\right)
\eea
with
\bea\label{orhotpsininit}
v_0=\psi_0+\e_0, \ \ \psi_0=\sum_{j=-\ell_0}^{-2}\sum_{M=0}^{M(j)}(a_{j,M})_0\phi_{j,2M}(Y)
\eea
and $\e_0$ satisfies the following orthogonality conditions
\be
\label{orthoeinit}
(\e_0,\phi_{j,2M})_{L^2_{\rho_Y}}=0, \ \ -\ell_0\leq j\leq -1, \ \ 0\leq M\leq M(j).
\ee
Let $K>0$ be a large enough universal constant such that in particular \eqref{eq:necessaryconditionK} holds true, and define
\be
\label{defchi}
\nu_K(z)=\frac{1}{1+z^{2K}}
\ee 
Let a large enough integer $q$ such that in particular \eqref{eq:necessaryconditionq} holds true, and pick $n\geq n(K)$ large enough and $s_0>s_0(n,K)$ large enough. Pick parameters $\lambda_0,b_0,(a_{j,M})_0$ and a profile $\e_0$ which satisfy the initial bounds:

\begin{itemize}
\item rescaled solution: 
\be
\label{scalingsmall}
\l_0 = e^{-\frac{s_0}2};
\ee
\item control of the $b$ parameter:
\be
\label{bparme}
b_0=\frac{1}{c_1s_0}
\ee
where the constant $c_1>0$ is given by \eqref{calculloi};

\vspace{0.2cm}

\item initial control of the unstable modes:
\be
\label{nvnono}
\sum_{j=-\ell_0}^{-2}\sum_{M=0}^{M(j)}|a_{j,M}(0)|^2\leq  \frac{1}{s_0^n};
\ee
\item initial control of the exponentially localized norm: 
\be
\label{poitwiseboundhtwoinitial}
\|\e_0\|_{H^2_{\rho_Y}}+\|\nabla\e_0\|_{L^{2q+2}_{\rho_Y}}<\frac{1}{s_0^{n}};
\ee
\item control of polynomially localized norms:
\be
\label{cekneoneoncoenoinitial}
\int \nu_K\e_0^2\rho_rdY\leq \frac{1}{s_0^{2K}}, \ \ \int \nu_K|\nabla\e_0|^{2q+2}\rho_rdY\leq\frac{1}{s_0^{2q+2K}};
\ee
\item initial control of the global $W^{1,2q+2}$ norm: 
\be
\label{controlsobolevinitial}
\|v_0\|_{W^{1,2q+2}}<\frac{1}{s_0}.
\ee
\end{itemize}

\begin{remark}
Note that the above properties of the initial data $u_0$ can be obtained by applying Lemma \ref{lemma:decompositionofinitialdata} to an initial data of the form
\be\label{eq:theinitialformoftheinitialdata}
u_0=\Phit_{\underline{b}_0}+w_0
\ee
where $\Phit_b$ has been constructed in Lemma \ref{lemmaapproximatelocalized} and where
\be\label{eq:theinitialformoftheinitialdatabis}
0<\underline{b}_0\ll 1\,\,\,\,\textrm{ and }\,\,\,\,\|w_0\|_{W^{1,2q+2}}+\|w_0\|_{H^2}\leq \underline{b}_0^{2n}.
\ee
Indeed, the decomposition \eqref{geomdeopcopmtiinit} \eqref{orhotpsininit} \eqref{orthoeinit} immediately follows from Lemma \ref{lemma:decompositionofinitialdata}. Then, we may choose $s_0$ as
$$s_0=\frac{1}{c_1b_0}$$
so that \eqref{bparme} holds true. In view of our assumptions on $w_0$, this yields in particular
$$\|w_0\|_{W^{1,2q+2}}+\|w_0\|_{H^2}\lesssim \frac{1}{s_0^{2n}},$$
and the estimates \eqref{nvnono} \eqref{poitwiseboundhtwoinitial} \eqref{cekneoneoncoenoinitial} \eqref{controlsobolevinitial} immediately follow from the bounds \eqref{eq:bound continuite decomposition} \eqref{eq:bound continuite decompositionbis}. Finally, we may always renormalize the initial data to enforce \eqref{scalingsmall}. 
\end{remark}

\noindent{\em Renormalized flow}. From a standard continuity in time argument, as long as the solution remains close to $\Phi$ up to scaling in $L^2_{\rho_Y}$, we may introduce the time dependent geometrical decomposition
\be
\label{geomdeopcopmti}
\left|\begin{array}{ll}u(t,x)=\frac{1}{\l(t)^{\frac{2}{p-1}}}U(s,Y), \ \ Y=\frac{x}{\l(t)}\\
 \  \  U=\Phit_{b(t)}+v, \ \ v=\psi+\e
 \end{array}\right.
\ee 
with 
\be
\label{orhotpsin}
\psi=\sum_{j=-\ell_0}^{-2}\sum_{M=0}^{M(j)}a_{j,M}(t)\phi_{j,2M}(Y)
\ee 
and 
\be
\label{orthoe}
(\e(t),\phi_{j,2M})_{L^2_{\rho_Y}}=0, \ \ -\ell_0\leq j\leq -1, \ \ 0\leq M\leq M(j).
\ee
The above decomposition is continuously differentiable with respect to time from standard parabolic regularizing effects. Consider the renormalized time 
\be
\label{defrescaledtime}
s(t)=\int_0^t\frac{d\tau}{\l^2(\tau)}+s_0,
\ee
then from \eqref{geomdeopcopmti}: $$\pa_sU-\lsl\Lambda U=\Delta U+U^p$$ which together with \eqref{definitionv}, \eqref{reconnectingequations} yields the $v$ equation:
\bea
\label{vequationbizarre}
\nonumber&& (b_s+bB(b))\pa_b\Phit_b-\left(\lsl+\frac 12-M(b)\right)\Lamdba \Phit_b+\pa_sv+\L v\\
&=& \Psit_b+\left(\lsl+\frac 12\right)\Lambda v+F(v)
\eea
where
\be
\label{defnlv}
F(v)=F_1+F_2,\ \ \left|\begin{array}{ll}F_1=p(\Phit_b^{p-1}-\Phi^{p-1})v, \\ F_2=(\Phit_b+v)^p-\Phit_b^p-p\Phit_b^{p-1}v.\end{array}\right.
\ee
We may equivalently develop $v=\psi+\e$ and obtain the $\e$ equation:
\be
\label{defequation}
\pa_s\e+\L\e=\Psit_b-\Mod+L(\e)+F(v)
\ee
where $\Mod$ encodes the modulation equations 
\bea
\label{modulation}
\nonumber\Mod&=&\sum_{j=-\ell_0}^{-2}\sum_{M=0}^{M(j)}\left[(a_{j,M})_s+(\l_j+M)a_{j,M}\right]\phi_{j,2M}-\left(\lsl+\frac 12\right)\Lambda \psi\\
 & -& \left(\lsl+\frac 12-M(b)\right)\Lambda \Phit_b+(b_s+bB(b))\pa_b\Phit_b
\eea
and we defined the linear error
\be
\label{linearLe}
L(\e)=\left(\lsl+\frac 12\right)\Lambda \e.
\ee
We claim the following bootstrap proposition.

\begin{proposition}[Bootstrap]
\label{bootstrap}
Given $q$ large enough satisfying in particular \eqref{eq:necessaryconditionq}, $K\geq K(q)$ large enough satisfying in particular \eqref{eq:necessaryconditionK}, $n\geq n(K,q)$ large enough and $s_0(n,K,q)$ large enough, then forall  $\lambda_0,b_0,\e_0$ satisfying \fref{scalingsmall}, \eqref{bparme}, \eqref{poitwiseboundhtwoinitial}, \eqref{cekneoneoncoenoinitial}, \eqref{controlsobolevinitial} and the orthogonality conditions \eqref{orthoeinit}, there exist $(a_{j,M}(0))_{-\ell_0\leq j\leq-2, 0\leq M\leq M(j)}$ satisfying  \eqref{nvnono} such that the solution starting from $u_0$ given by \fref{geomdeopcopmtiinit}, decomposed according to \eqref{geomdeopcopmti} satisfies for all $s\geq s_0$:
\begin{itemize}
\item control of the scaling:
\be
\label{controlsclaing}
0<\l(s)< e^{-\frac s4};
\ee
\item control of the $b$ parameter:
\be
\label{bparmebis}
\frac{1}{10c_1s}<b(s)<\frac{10}{c_1s};
\ee
\item control of the unstable modes:
\be
\label{controlunstable}
\sum_{j=-\ell_0}^{-2}\sum_{M=0}^{M(j)}|a_{j,M}(s)|^2\leq \frac{1}{s^{n}};
\ee
\item control of the exponentially localized norm: 
\be
\label{poitwiseboundhtwo}
\|\e(s)\|_{H^2_\rho}<\frac{1}{s^{\frac{n}{2}}}
\ee
and 
\be
\label{neneoneovneoneov}
\|\nabla\e\|_{L^{2q+2}_{\rho_Y}}<\frac{1}{s^{\frac n2}};
\ee
\item control of polynomially localized norms:
\be
\label{cekneoneoncoeno}
\int \nu_K|\e(s)|^2\rho_rdY\leq \frac{1}{s^{K+1}}, \ \ \int \nu_K|\nabla\e(s)|^{2q+2}\rho_rdY\leq\frac{1}{s^{2q+K+1}};
\ee
\item control of the global $W^{1,2q+2}$ norm: 
\be
\label{controlsobolev}
\|v(s)\|_{W^{1,2q+2}}<\frac{1}{s^{\delta_q}}
\ee
for some small enough $\delta_q>0$.
\end{itemize}
\end{proposition}

Proposition \ref{bootstrap} is the heart of the analysis, and the corresponding solutions are easily shown to satisfy the conclusions of Theorem \ref{thmmain}. The strategy of the proof follows \cite{martelmulti,MRR}: we prove Proposition \ref{bootstrap} by contradiction using a topological argument \`a la Brouwer: given $\lambda_0,b_0,\e_0$ satisfying \fref{scalingsmall}, \eqref{bparme}, \eqref{poitwiseboundhtwoinitial}, \eqref{cekneoneoncoenoinitial}, \eqref{controlsobolevinitial}, \eqref{orthoeinit}, we assume that for all $(a_{j,M}(0))_{-\ell_0\leq j\leq-2,\, 0\leq M\leq M(j)}$ satisfying \eqref{nvnono}, the exit time 
\bea
\label{defexittime}
\nonumber s^*=\sup&\{&s\geq s_0\ \ \mbox{such that}\ \ \eqref{controlsclaing}, \eqref{bparmebis}, \eqref{controlunstable}, \eqref{poitwiseboundhtwo}, \eqref{cekneoneoncoeno}, \eqref{controlsobolev}\\
& &  \text{holds} \ \text{on} \ [s_0,s)\}
\eea
is finite 
\be
\label{assumptioncontradiction}
s^*<+\infty
\ee
and look for a contradiction for $s_0\geq s_0(n,K,q)$ large enough. From now on, we therefore study the flow on $[s_0,s^*]$ where  \eqref{controlsclaing}, \eqref{bparmebis}, \eqref{controlunstable}, \eqref{poitwiseboundhtwo}, \eqref{cekneoneoncoeno}, \eqref{controlsobolev} hold. Using a bootstrap method we show that the bounds  \eqref{controlsclaing}, \eqref{bparmebis}, \eqref{poitwiseboundhtwo}, \eqref{cekneoneoncoeno},  \eqref{controlsobolev} can be improved, implying that at time $s^*$ necessarily the unstable modes have grown and \fref{controlunstable} reaches its boundary. Since $0$ is a linear repulsive equilibrium for these modes, this will contradict Brouwer fixed point theorem.


\subsection{Modulation equations}


We now compute the modulation equations which describe the time evolution of the parameters. They are computed in the self-similar zone, and involve the $\rho$ weighted norm.

\begin{lemma}[Modulation equations]
\label{lemmamodulation}
There holds the modulation equations:
\bea
\label{controlMOd}
\nonumber &&\left|\lsl+\frac12-M(b)\right|+|b_s+bB(b)|+\sum_{j=-\ell_0}^{-2}\sum_{M=0}^{M(j)}|(a_{j,M})_s+(\l_j+M)a_{j,M}|\\
&\lesssim & b^{n+1}+b\left(\|\e\|_{L^2_{\rho_Y}}+\sum_{j=-\ell_0}^{-2}\sum_{M=0}^{M(j)} |a_{j,M}|\right).
\eea
\end{lemma}

\begin{proof} This lemma follows from the choice of orthogonality conditions \eqref{orthoe} and the explicit properties of the refined reconnecting profile $\Phit_b$. The control of the nonlinear term relies in an essential way on \eqref{controlsobolev} which from Sobolev implies for $q$ large enough the $L^\infty$ smallness
\be
\label{poitwisebound}
\|v\|_{L^{\infty}}\lesssim \|v\|_{W^{1,2q+2}}\lesssim \frac{1}{s^{\delta_q}}\ll1.
\ee
We take the $L^2_{\rho_Y}$ scalar product of \eqref{defequation} with $\phi_{j,2M}$ and compute from \eqref{orthoe}: $$(\Mod,\phi_{j,2M})_{L^2_{\rho_Y}}=(\Psit_b,\phi_{j,2M})_{L^2_{\rho_Y}}+(L(\e)+F(v),\phi_{j,2M})_{L^2_{\rho_Y}}.$$ 
The error term in controlled from \eqref{eq:decompoistionPsitbinPsitb0andG} \eqref{esterrorinomegabis} thanks to the space localization of the $\rho_YdY$ measure : $$|(\Psit_b,\phi_{j,2M})_{L^2_{\rho_Y}}|\lesssim b^{n+1}.$$
The linear term is estimated by integration by parts 
$$|(L(\e),\phi_{j,2M})_{L^2_{\rho_Y}}|\lesssim \left|\lsl+\frac 12\right|\|\e\|_{L^2_{\rho_Y}}.$$ For the nonlinear term, we recall \eqref{defnlv}. We estimate:
$$|\pa_b\Phi_b|=\left|-\frac{z^2}{2(1+bz^2)}\frac{1}{\mu_b^{\frac2{p-1}}}\Lambda\Phi\left(\frac{r}{\mu_b}\right)\right|\lesssim |z|^2, \ \ |\pa_b\tilde{v}_b|\lesssim 1$$ which using $\|\Phit_b\|_{L^{\infty}}\lesssim 1$ implies the pointwise bound 
\be
\label{pointwiseboundlinearerror}
|\Phit_b^{p-1}-\Phi^{p-1}|\lesssim \int_0^b\left |\Phit_b^{p-2}\pa_b\Phit_b\right|db\lesssim b(1+|z|^2)
\ee
and hence
$$\left|(F_1(v),\phi_{j,2M})_{L^2_{\rho_Y}}\right|\lesssim b\left|(v,(1+|z|^2)\phi_{j,2M})_{L^2_{\rho_Y}}\right|\lesssim b\left(\|\e\|_{L^2_{\rho_Y}}+\sum |a_{j,M}|\right).$$
For the remaining nonlinear term, we use the rough $L^{\infty}$ bound $\|v\|_{L^\infty}+\|\Phit_b\|_{L^{\infty}}\leq 1$ and the confining measure: 
\bee
&&|(F_2(v),\phi_{j,2M})_{L^2_{\rho_Y}}|\lesssim \int (|v|^2+|v|^p)|\phi_{j,2M}|\rho_YdY\lesssim \int |v|^2|\phi_{j,2M}|\rho_YdY\\
&\lesssim & b\int |v||\phi_{j,2M}|\rho_YdY\lesssim b\|v\|_{L^2_{\rho_Y}}\lesssim  b\left(\|\e\|_{L^2_{\rho_Y}}+\sum|a_{j,M}|\right)
\eee
where we used the fact that $v=\psi+\e$ and the rough bound 
\be
\label{roughboundmodes}
\|v\|_{L^2_{\rho_Y}}\leq\|\e\|_{L^2_{\rho_Y}}+\sum |a_{j,M}|\leq b
\ee
which follows from \eqref{bparmebis} \eqref{controlunstable} \eqref{poitwiseboundhtwo}. We therefore have obtained the following identity:
\be
\label{firstidentity}
\left|(\Mod,\phi_{j,2M})_{L^2_{\rho_Y}}\right|\lesssim b^{n+1}+\left(\left|\lsl+\frac 12-M(b)\right|+b\right)\|\e\|_{L^2_{\rho_Y}}+b\sum|a_{j,M}|.
\ee
We now compute the lhs of \eqref{firstidentity} for the various values of $j$.\\

\noindent{\underline{$a_{j,M}$ terms, $j\leq -2$}}. First observe from \eqref{definitionv}, \eqref{degeneracy} the bounds 
\be
\label{ceoneoneoc}
\| \nabla_Y^k\Lambda \tilde{v}_b\|_{L^2}+\|\nabla_Y^k\pa_b\tilde{v}_b\|_{L^2_{\rho_Y}}\lesssim b,\ \ k=0, 1, 2,
\ee
which together with the computations 
\bee
\nabla\Phi_b = \frac{1-bz^2}{1+bz^2}\frac{1}{\mu_b^{\frac{2}{p-1}}}\Lambda\Phi\left(\frac{r}{\mu_b}\right),\,\,\,\, \pr_b\Phi_b = \frac{z^2}{2\mu_b^2}\frac{1}{\mu_b^{\frac{2}{p-1}}}\Lambda\Phi\left(\frac{r}{\mu_b}\right)
\eee
yields
\be
\label{estpourmod}
\| \nabla_Y^k(\Lambda \Phit_b-\Lambda \Phi)\|_{L^2_{\rho_Y}}+\left\| \nabla_Y^k\left(\pa_b\Phit_b+\frac{1}{2}z^2\Lambda \Phi\right)\right\|_{L^2_{\rho_Y}}\lesssim b,\ \ k=0, 1, 2.
\ee
Hence, we have in particular 
$$ (\pa_b\Phit_b,\phi_{j,2M})_{L^2_{\rho_Y}}=O(b), \ \ (\Lambda \Phit_b,\phi_{j,2M})_{L^2_{\rho_Y}}=O(b).$$
We conclude from \eqref{modulation} using the orthonormality of eigenfunctions, separation of variables and the rough bound \eqref{roughboundmodes}:
\bea
\label{estuninitial}
(\Mod,\phi_{j,2M})_{L^2_{\rho_Y}} &=&\left[(a_{j,M})_s+(\l_j+M)a_{j,M}\right]\|\phi_{j,2M}\|_{L^2_{\rho_Y}}^2\\
\nonumber &+& O\left[b\left(\left|\lsl+\frac 12-M(b)\right|+|a_{j,M}|+|b_s+bB(b)|\right)\right].
\eea

\noindent{\underline{Scaling terms}}. We compute from \eqref{estpourmod} :
$$
(\Lambda \Phit_b,\Lamdba \Phi)_{L^2_{\rho_Y}}=\|\Lambda \Phi\|^2_{L^2_{\rho_Y}}+O(b)$$
and hence:
\bea
\label{estuninitialbis}
(\Mod,\Lambda \Phi)_{L^2_{\rho_Y}}&=&-\left(\lsl+\frac 12-M(b)\right)\left[\|\Lambda \Phi\|^2_{L^2_{\rho_Y}}+O(b)\right]\\
\nonumber &+& O\left[|b_s+bB(b)|+b\left(\left|\lsl+\frac 12-M(b)\right|+\sum|a_{j,M}|\right)\right].
\eea

\noindent{\underline{$b$ equation}}. We compute from \eqref{estpourmod}:
\bee
&&(\Lambda \Phit_b,(z^2-2)\Lambda\Phi)_{L^2_{\rho_Y}}=O(b)\\
&&(\pa_b\Phit_b,(z^2-2)\Lambda\Phi)_{L^2_{\rho_Y}}=-\frac12\|(z^2-2)\Lambda \Phi\|_{L^2_{\rho_Y}}^2+O(b)
\eee
from which using the orthogonality of eigenfunctions:
\bea
\label{estuninitialbisbis}
\nonumber &&(\Mod,(z^2-2)\Lambda \Phi)_{L^2_{\rho_Y}}=-\frac12 \|(z^2-2)\Lambda \Phi\|_{L^2_{\rho_Y}}^2 (1+O(b))(b_s+bB(b))\\
&+& O\left[b\left(\left|\lsl+\frac 12-M(b)\right|+\sum|a_{j,M}|\right)\right].
\eea

\noindent{\underline{Conclusion}}.
Injecting \eqref{estuninitial}, \eqref{estuninitialbis}, \eqref{estuninitialbisbis} into \eqref{firstidentity} yields \eqref{controlMOd}.
\end{proof}


\subsection{Inner $H^2$ bounds with exponential localization}


We now turn to the control of the flow in exponentially weighted norms which is an elementary consequence of  the spectral gap estimate \eqref{spectralgapestimate}, the dissipative structure of the flow,  the $L^\infty$ bound \eqref{poitwisebound} to control the non linear term and the explicit form of the refined reconnecting $\Phit_b$ profiles which generate the leading order  error term.

\begin{lemma}[Lyapunov control of exponentially weighed norms]
\label{boundsexpo}
There holds the pointwise differential bounds:
\be
\label{differentialcontrolnorm}
\frac{d}{ds}\|\e\|_{L^2_{\rho_Y}}^2+c\|\e\|_{H^1_{\rho_Y}}^2 \lesssim b^{2n+2}+ (\|v\|^2_{L^{\infty}}+b^2)\sum|a_{j,M}|^2,
\ee
\be
\label{woneqlocbound}
\frac{d}{ds}\|\nabla\e\|^{2q+2}_{L^{2q+2}_{\rho_Y}}+c\|\nabla\e\|^{2q+2}_{L^{2q+2}_{\rho_Y}}\lesssim \|\e\|_{H^1_{\rho_Y}}^{2q+2}+\sum |a_{j,M}|^{2q+2}+ b^{(2q+2)(n+1)},
\ee
\bea
\label{differentialcontrolnormbisbis}
&&\frac{d}{ds}\|\mathcal{L}_Y\e\|_{L^2_{\rho_Y}}^2+c\|\mathcal{L}_Y\e\|_{H^1_{\rho_Y}}^2\lesssim  b^{2n+2}+(b+\|v\|^2_{L^\infty})\|\e\|_{H^1_{\rho_Y}}^2\\
\nonumber &+& (b^2+\|v\|_{L^\infty}^2)\sum |a_{j,M}|^2+e^{-\frac{c}{\sqrt{b}}}\int \frac{\e^2+|\nabla_Y\e|^2}{1+|z|^{2K}}\rho_r dY
\eea
for some universal constant $c>0$.
\end{lemma}

\begin{proof} 
{\bf step 1} $L^2$ exponential bound. We compute from \eqref{defequation}:
\bea
\label{vnbeovneoneo}
\nonumber &&\frac 12\frac{d}{ds}\|\e\|_{L^2_{\rho_Y}}^2=(\e,\pa_s\e)_{L^2_{\rho_Y}}\\
&=&-(\mathcal{L}\e,\e)_{L^2_{\rho_Y}}+(\Psit_b+L(\e)-\Mod+F(v),\e)_{L^2_{\rho_Y}}
\eea
and estimate all terms in the above identity.\\
We start with the nonlinear term \eqref{defnlv}. Recall the variance bound\footnote{see for example \cite{CRS}, Appendix A.}  
\be
\label{wegithedrho}
\|Yu\|_{L^2_{\rho_Y}}\lesssim \|u\|_{H^1_{\rho_Y}}
\ee
which together with the pointwise bound \eqref{pointwiseboundlinearerror}
ensures $$|([\Phi_b^{p-1}-\Phi^{p-1}]\e,\e)_{L^2_{\rho_Y}}\lesssim b((1+|z|^2)\e,\e)_{L^2_{\rho_Y}}\lesssim b\|\e\|_{H^1_{\rho_Y}}^2.$$ We now estimate using the rough $L^\infty$ bound $\|v\|_{L^{\infty}}\ll 1$:
\bee
|(F_2(v),\e)_{\rho}|&\lesssim & \int |\e|v^2\rho_Y dY \leq\delta\int|\e|^2\rho+C_\delta \int |v|^4\rho_Y dY\\
& \leq & \delta \|\e\|_{L^2_{\rho_Y}}^2+C_\delta\|v\|^2_{L^{\infty}}\int |v|^2\rho_YdY\\
& \lesssim & \delta \|\e\|_{L^2_{\rho_Y}}^2+\|v\|^2_{L^{\infty}}\sum|a_{j,M}|^2.
\eee
 To estimate the $L$ term, we use the rough bound from \eqref{controlMOd}:
 \be
 \label{roughboundparameters}
 \left|\lsl+\frac12\right|+|b_s+bB(b)|+|(a_{j,M})_s-(\l_j+M)a_{j,M}|\lesssim b
 \ee 
 which implies using  \eqref{integrationbyparts}, \eqref{wegithedrho}:
\be
\label{comutareurwirhgt}
|(\e,L(\e))_{L^2_{\rho_Y}}|\lesssim b|(\e,\Lambda \e)_{L^2_{\rho_Y}}|\lesssim b\|(1+|Y|)\e\|_{L^2_{\rho_Y}}\lesssim b\|\e\|_{H^1_\rho}^2.
\ee
The leading order term $\Psit_b$ term is estimated in brute force from \eqref{eq:decompoistionPsitbinPsitb0andG} \eqref{esterrorinomegabis} using the exponential localization of the measure: 
$$\left|(\e,\Psit_b)_{L^2_{\rho_Y}}\right|\lesssim b^{n+1}\|\e\|_{L^2_{\rho_Y}}.$$
To control the modulation parameters, we use \eqref{estpourmod}, \eqref{orthoe}, \eqref{controlMOd} to estimate:
\bee
|(\e, \Mod)|&\lesssim& b\left[\sum |a_{j,M}|+b^{n+1}+b\|\e\|_{L^2_{\rho_Y}}\right]\|(1+|Y|)\e\|_{L^2_{\rho_Y}}\\
& \lesssim & \de\|\e\|^2_{H^1_{\rho_Y}}+c_\de b^{2n+4}+c_\de b^2\sum |a_{j,M}|^2.
\eee
Injecting the collection of above bounds into \eqref{vnbeovneoneo} and using the spectral gap estimate \eqref{spectralgapestimate} with the choice of orthogonality conditions \eqref{orthoe} yields \eqref{differentialcontrolnorm}.\\

\noindent{\bf step 2} $\dot{W}^{1,2q+2}$ exponential bound. Let $q$ be a large enough integer. Let 
$$\e_i=\pa_{i}\e,\ \ i=1,2,3,4,$$ 
then from \eqref{defequation}:
\be
\label{eqei}
\pa_s\e_i+(\L+1)\e_i=\pa_i\left[\Psit_b-\Mod+L(\e)+F(v)\right]+p(p-1)\Phi^{p-2}\pr_i\Phi \ep.
\ee
We then compute:
\bee
&&\frac{1}{2q+2}\frac{d}{ds}\int \e_i^{2q+2}\rho_YdY=\int \e_i^{2q+1}\pa_s\e_i\\
&=& -\left((\L+1)\e_i, \e_i^{2q+1}\right)_{L^2_{\rho_Y}}+\left(\e_i^{2q+1},\pa_i\left[\Psit_b-\Mod+L(\e)+F(v)\right]\right)_{L^2_{\rho_Y}}\\
&& + \left(\e_i^{2q+1},p(p-1)\Phi^{p-2}\pr_i\Phi \ep\right)_{L^2_{\rho_Y}}
\eee
and estimate all terms in the above identity.

We integrate by parts to compute:
\bee
&&\left(\left(\Delta-\frac 12 Y\cdot\nabla\right)\e_i,\e_i^{2q+1}\right)_{L^2_{\rho_Y}}=\int \frac{1}{\rho_Y}\nabla \cdot(\rho_Y\nabla \e_i)\e_i^{2q+1}\rho_YdY\\
&=& -(2q+1)\int \e_i^{2q}|\nabla_Y\e_i|^{2}\rho_YdY=-\frac{2q+1}{(q+1)^2}\int |\nabla_Y(\e_i^{q+1})|^{2}\rho_YdY.
\eee
 We apply the spectral gap estimate \eqref{spectralgapestimate} to $\e_i^{q+1}$ and conclude that there exists $c>0$, and for all $A>0$ large enough, there exists $C_A$ such that 
 \bee
 \int |\nabla_Y(\e_i^{q+1})|^{2}\rho_YdY &\geq& c\int |\nabla(\e_i^{q+1})|^{2}\rho_YdY+  A\int (\e_i^{q+1})^{2}\rho_YdY\\
 &&-C_A\sum_{j,M\leq j(A),M(A)} \left(\e_i^{q+1},\phi_{j,2M}\right)_{L^2_{\rho_Y}}^2,
 \eee
 where $j,M\leq j(A),M(A)$ are the indices corresponding to all eigenvalues $\mu_{j,2M}$ of $\L_Y$ that satisfy $\mu_{j,2M}\leq A$.  Hence choosing $A$ large enough compared to $\|\Phi\|_{L^{\infty}}$, we infer
 \bee
 -\left((\L+1)\e_i, \e_i^{2q+1}\right)_{L^2_{\rho_Y}}&\leq& -c\int |\nabla(\e_i^{q+1})|^{2}\rho_YdY-\frac A2 \int (\e_i^{q+1})^{2}\rho_YdY\\
 &&+C_A\sum_{j,M\leq j(A),M(A)} \left(\e_i^{q+1},\phi_{j,2M}\right)_{L^2_{\rho_Y}}^2.
 \eee
 We now estimate using H\"older and the polynomial growth of eigenmodes $|\phi_{j,2M}|\lesssim |Y|^{c(j,M)}$: 
 \bee
&& \left(\e_i^{q+1},\phi_{j,2M}\right)^2_{L^2_{\rho_Y}}\lesssim \left(\int |\e_i|^{2q}|\phi_{j,2M}|^2\rho_YdY\right)\left(\int |\e_i|^2\rho_YdY\right)\\
 &\lesssim &\left(\int |\e_i|^{2q+2}\rho_YdY\right)^{\frac{2q}{2q+2}}\int |\e_i|^2\rho_YdY\\
& \leq& \delta\int |\e_i|^{2q+2}\rho_YdY+c_\delta\left(\int |\e_i|^2\rho_YdY\right)^{q+1}.
 \eee
 and hence, for $\de$ small enough compared to $C_A$, $j(A)$ and $M(A)$, we infer
 \be\label{eq:intermediaryestimateusefulinvolvingAcoercivity}
 -\left((\L+1)\e_i, \e_i^{2q+1}\right)_{L^2_{\rho_Y}}\leq -c\int |\nabla(\e_i^{q+1})|^{2}\rho_YdY-\frac A4 \int (\e_i^{q+1})^{2}\rho_YdY+C_A\|\e\|_{H^1_{\rho_Y}}^{2q+2}.
 \ee 
 
 The leading order error term is controlled from \eqref{esterrorinomegabis}:
 $$\left|\left(\e_i^{2q+1},\pa_i\Psit_b\right)_{L^2_{\rho_Y}}\right|\lesssim \int |\e_i|^{2q+2}\rho_YdY+\int|\pa_i\Psit_b|^{2q+2}\rho_YdY\lesssim \int |\e_i|^{2q+2}\rho_YdY+b^{(2q+2)(n+1)}.$$
 
 We integrate by parts and use \eqref{estimportante} to estimate:
 $$|(\e_i^{2q+2},\pa_i\Lambda \e)_{L^2_{\rho_Y}}|\lesssim \int (1+|Y|^2)\e_i^{2q+2}\rho_YdY\lesssim \int \left[\e_i^{2q+2}+|\nabla_Y(\e_i^{q+1})|^2\right]\rho_YdY$$
 and hence from \eqref{roughboundparameters}:
 $$|(\e_i^{2q+1},L(\e))_{L^2_{\rho_Y}}|\lesssim b\int \left[\e_i^{2q+2}+|\nabla_Y(\e_i^{q+1})|^2\right]\rho_YdY.$$
 
 Also, we have
 \bee
\left| \left(\e_i^{2q+1},p(p-1)\Phi^{p-2}\pr_i\Phi \ep\right)_{L^2_{\rho_Y}}\right| \lesssim  \int |\e_i|^{2q+1}|\e|\rho_YdY\lesssim  \int |\e_i|^{2q+2}\rho_YdY.
 \eee 
  
 We now turn to the control of the nonlinear term.  We first estimate using $\|\Phit_b\|_{L^{\infty}}+\|\nabla \Phit_b\|_{L^\infty}\lesssim 1$,  H\"older and the polynomial growth of $\psi$: \bee
&& |(\pa_iF_1(v),\e_i^{2q+1})_{L^2_{\rho_Y}}|\lesssim \int |\e_i|^{2q+1}(|v|+|\nabla v|)\rho_YdY\\
 &\lesssim &\int |\e_i|^{2q+2}\rho_YdY+\int (|v|+|\nabla v|)^{2q+2}\rho_YdY \\
 &\lesssim&  \int |\e_i|^{2q+2}\rho_YdY+\sum |a_{j,M}|^{2q+2}.
 \eee
  
We now compute 
 \bee
\nabla F_2(v)&=& p\nabla_Y v\left[(\Phit_b+v)^{p-1}-\Phit_b^{p-1}\right]\\
\nonumber & +& p\nabla_Y \Phit_b\left[(\Phit_b+v)^{p-1}-\Phit_b^{p-1}-(p-1)\Phit_b^{p-2}v\right]
\eee
and estimate by homogeneity with the $L^\infty$ bound \eqref{poitwisebound}: 
 \be
 \label{pointwisenls2}
 |F_2(v)|\lesssim |v|^2,\,\,\,\, |\nabla_YF_2(v)|\lesssim |\nabla_Y v||v|+|v|^2\lesssim |\nabla_Y v|+|v|
 \ee 
 and hence the same bound as above:
 \bee
  |(\pa_iF_1(v),\e_i^{2q+1})_{L^2_{\rho_Y}}|\lesssim\int |\e_i|^{2q+1}(|v|+|\nabla v|)\rho_YdY\lesssim  \int |\e_i|^{2q+2}\rho_YdY+\sum |a_{j,M}|^{2q+2}.
  \eee
   The collection of above bounds for $i=1,2, 3, 4$ yields \eqref{woneqlocbound} provided the constant $A$ in \eqref{eq:intermediaryestimateusefulinvolvingAcoercivity}  has been chosen large enough.\\
  
\noindent{\bf step 3} $\dot{H}^2$ exponential bound. Let $$\e_{(2)}=\mathcal{L}_Y\e,$$ then $\e_{(2)}$ satisfies the orthogonality conditions \eqref{orthoe} and the equation from \eqref{defequation}:
\be
\label{eqetwo}
\pa_s\e_{(2)}+\mathcal{L}_Y\e_{(2)}=\mathcal L_Y\left[\Psit_b-\Mod+L(\e)+F(v)\right]
\ee
and hence
\be
\label{cneknenonee}
\frac{1}{2}\frac{d}{ds}\|\e_{(2)}\|_{L^2_{\rho_Y}}^2=\left(-\mathcal{L}_Y\e_{(2)}+\mathcal L_Y\left[\Psit_b-\Mod+L(\e)+F(v)\right],\e_{(2)}\right)_{L^2_{\rho_Y}}.
\ee 
The main forcing term is estimated in brute force using \eqref{eq:decompoistionPsitbinPsitb0andG} \eqref{esterrorinomegabis}:
$$(\mathcal L_Y(\Psit_b),\e_{(2)})_{L^2_{\rho_Y}}\lesssim b^{n+1}\|\e_{(2)}\|_{L^2}\leq c_\delta b^{2n+2}+\delta\|\e_{(2)}\|^2_{H^1_{\rho_Y}}.$$
The $\Mod$ terms are controlled using \eqref{estpourmod}, \eqref{orthoe}, \eqref{controlMOd} which yield:
\bee
(\mathcal L_Y\Mod,\e_{(2)})_{L^2_{\rho_Y}}&\lesssim & b\left[\sum |a_{j,M}|+b^{n+1}+b\|\e\|_{L^2_{\rho_Y}}\right]\|(1+Y)\e_{(2)}\|_{L^2_{\rho_Y}}\\
& \lesssim & 
\de\|\e_{(2)}\|^2_{H^1_{\rho_Y}}+c_\de b^{2n+4}+c_\de b^2\sum |a_{j,M}|^2.
\eee
For the $L(\e)$ term, we use the commutator relation 
\be
\label{commrealtisn}
[\Delta_Y, \Lambda_Y]=2\Delta_Y
\ee
to compute 
\bee
[\mathcal{L}_Y,\Lambda_Y]&=& [-\Delta_Y+\Lambda_Y-p\Phi^{p-1},\Lambda_Y]=-2\Delta_Y+p(p-1)\Phi_n^{p-2}r\pa_r\Phi\\
&=& 2(\mathcal L_Y-\Lambda_Y +p\Phi^{p-1})+p(p-1)\Phi^{p-2}r\pa_r\Phi
\eee
from which using \eqref{integrationbyparts} \eqref{comutareurwirhgt}, \eqref{wegithedrho} and $\|\Phi\|_{L^{\infty}}+\|\Lambda\Phi\|_{L^{\infty}}\lesssim 1$:
\bee
|(\e_{(2)},\mathcal{L}_Y\Lambda_Y \e)_{L^2_{\rho_Y}}|&=& \left|(\e_{(2)},[\mathcal{L}_Y,\Lambda_Y]\e)_{L^2_{\rho_Y}}+(\e_{(2)},\Lambda_Y \e_{(2)})_{L^2_{\rho_Y}}\right|\\
& \lesssim &\|\e_{(2)}\|_{H^1_{\rho_Y}}^2+|(\e_{(2)},\Lambda \e)_{L^2_{\rho_Y}}|+\|\e\|_{L^2_{\rho_Y}}^2\lesssim \|\e_{(2)}\|_{H^1_{\rho_Y}}^2+\|\e\|_{H^1_{\rho_Y}}^2
\eee 
and hence from \eqref{roughboundparameters}:
$$
|(\e_{(2)},\mathcal{L}_YL(\e) )_{L^2_{\rho_Y}}|\lesssim b\left(\|\e_{(2)}\|_{H^1_{\rho_Y}}^2+\|\e\|_{H^1_{\rho_Y}}^2\right).
$$
It remains to estimate the nonlinear term. We first integrate by parts since $\mathcal{L}_Y$ is self adjoint for $(\cdot,\cdot)_{L^2_{\rho_Y}}$:
$$
|(\mathcal{L}_YF,\e_{(2)})_{L^2_{\rho_Y}}| = \left|(\nabla F,\nabla \e_{(2)})_{L^2_{\rho_Y}}+\left(\frac{2}{p-1}F-p\Phi^{p-1}F, \e_{(2)}\right)_{L^2_{\rho_Y}}\right|.$$
We recall the decomposition \eqref{defnlv}. For the first term, we need to deal with the fact that the difference $\Phit_b-\Phi$ is not $L^{\infty}$ small for $|Z|\gtrsim 1$. We first estimate pointwise using \eqref{degeneracy}
$$\left|\pa_b\Phi_b\right|=\left|\frac{1}{2b}Z\pa_ZG\right|=\left|-\frac{1}{b}\frac{Z^2}{\mu^{\frac {2}{p-1}+1}}\Lambda\left(\frac{r}{\mu(Z)}\right)\right|\lesssim \frac{Z^2}{b}=z^2, \ \ |\pa_b\tilde{v_b}|\lesssim b
$$
and similarly for higher derivatives, and hence the pointwise bound 
\be
\label{pintwisepa}
|\pa_b\Phit_b|+|\nabla_Y\pa_b\Phit_b|\lesssim 1+z^2.
\ee
This implies
\be
\label{enkvnenveonoe|}
|\nabla^k_Y(\Phit_b^{p-1}-\Phi^{p-1})|=\left|(p-1)\int_0^b\nabla_Y\left(\Phit_b^{p-2}\pa_b\Phit_b\right)db\right|\lesssim b(1+z^2), \ \ k=0,1.
\ee

We first estimate:
$$
\left|\left(\left(\frac{2}{p-1}-p\Phi^{p-1}\right)F_1, \e_{(2)}\right)_{L^2_{\rho_Y}}\right|\lesssim b\int (1+|Y|^2)|\e_{(2)}|^2\rho_YdY\lesssim b\|\e_{(2)}\|_{H^1_{\rho_Y}}^2.$$
Next:
\bee
&&|((\Phit_b^{p-1}-\Phi^{p-1})\nabla_Y v,\nabla_Y\e_{(2)})_{L^2_{\rho_Y}}\\
&\leq& \delta \|\nabla_Y \e_{(2)}\|_{L^2_Y}^2+c_\delta b^2\sum |a_{j,M}|^2+c_\delta\int |\Phit_b^{p-1}-\Phi^{p-1}|^2|\nabla_Y\e|^2
\eee
and
\bee
&&|(v\nabla_Y(\Phit_b^{p-1}-\Phi^{p-1}),\nabla_Y\e_{(2)})_{L^2_{\rho_Y}}|\\
&\leq &\delta \|\nabla_Y \e_{(2)}\|_{L^2_Y}^2+c_\delta b^2\sum |a_{j,M}|^2+c_\delta\int |\nabla(\Phit_b^{p-1}-\Phi^{p-1})|^2|\e|^2.
\eee
We split the last integral in two parts using \eqref{enkvnenveonoe|}:
\bee
&& \int\Big(|\nabla(\Phit_b^{p-1}-\Phi^{p-1})|^2\e^2+ |\Phit_b^{p-1}-\Phi^{p-1}|^2|\nabla_Y\e|^2\Big)\rho_YdY\\
&\lesssim & \int_{|z|\leq \frac{1}{b^{\frac{1}{4}}}}\Big(|\nabla(\Phit_b^{p-1}-\Phi^{p-1})|^2\e^2+ |\Phit_b^{p-1}-\Phi^{p-1}|^2|\nabla_Y\e|^2\Big)\rho_YdY\\
&&+e^{-\frac{c}{\sqrt{b}}}\int \frac{\e^2+|\nabla_Y\e|^2}{1+|z|^{2K}}r^2e^{-\frac{r^2}{2}}drdz\\
& \lesssim & b\|\e\|_{H^1_{\rho_Y}}^2+e^{-\frac{c}{\sqrt{b}}}\int \frac{\e^2+|\nabla_Y\e|^2}{1+|z|^{2K}}\rho_r dY
\eee
and hence the control of the first nonlinear term:
$$
|(\nabla F_1,\nabla \e_{(2)})_{L^2_{\rho_Y}}|\leq \delta \|\nabla \e_{(2)}\|_{L^2_{\rho_Y}}^2+c_\de b^2\sum a_{j,M}^2+e^{-\frac{c}{\sqrt{b}}}\int \frac{\e^2+|\nabla_Y\e|^2}{1+|z|^{2K}}\rho_r dY.
$$
For the second nonlinear term, we compute explicitly  
\bee
\nabla F_2(v)&=& p\nabla_Y v\left[(\Phit_b+v)^{p-1}-\Phit_b^{p-1}\right]\\
\nonumber & +& p\nabla_Y \Phit_b\left[(\Phit_b+v)^{p-1}-\Phit_b^{p-1}-(p-1)\Phit_b^{p-2}v\right].
\eee
 We estimate by homogeneity with the $L^\infty$ bound \eqref{poitwisebound}: 
 \be
 \label{pointwisenls2bis}
 |F_2(v)|\lesssim |v|^2,\,\,\,\, |\nabla_YF_2(v)|\lesssim |\nabla_Y v||v|+|v|^2
 \ee 
 and hence the bound using \eqref{poitwisebound} again:
\bee
&& |(\nabla F_2(v),\nabla \e_{(2)})_{L^2_{\rho_Y}}|+\left|\left(\frac{2}{p-1}F_2(v)-p\Phi_b^{p-1}F_2(v), \e_{(2)}\right)_{L^2_{\rho_Y}}\right|\\
&\lesssim& \int\left[|v||\nabla_Yv|+|v|^2\right]|\nabla\e_{(2)}|\rho_Y dY  +\int |\ep_{(2)}||v|^2\rho_Y dY\\
&\leq& \delta\|\e_{(2)}\|_{H^1_{\rho_Y}}^2+C_\delta\left[\int|v|^2|\nabla_Y v|^2\rho_Y dY+\int|v|^4\rho_Y dY\right]
\leq  \delta\|\e_{(2)}\|_{H^1_{\rho_Y}}^2+C_\delta\|v\|_{L^\infty}^2\|v\|_{H^1_{\rho_Y}}^2\\
& \leq & \delta\|\e_{(2)}\|_{H^1_{\rho_Y}}^2+C_\delta\|v\|_{L^\infty}^2\left[\|\e\|_{H^1_{\rho_Y}}^2+\sum|a_{j,M}|^2\right].
\eee
The collection of above bounds together with the spectral gap estimate \eqref{spectralgapestimate} and the orthogonality conditions \eqref{orthoe} injected into \eqref{cneknenonee} yields \eqref{differentialcontrolnormbis}.
\end{proof}


\subsection{Inner $W^{1,2q+2}$ bounds with polynomial localization in $z$}


The bounds of Lemma \ref{boundsexpo} rely in an essential way on the spectral gap estimate \eqref{spectralgapestimate} which demands a Gaussian like localization measure. Once these bounds are known, they can be turned into {\it polynomially} weighted bounds provided the weight is strong enough, {\it and the approximate solution of Lemma \ref{lemmaapproximate} has been developed to a sufficiently high order}. 

\begin{lemma}[Lyapunov control of polynomially weighted norms]
\label{boundspoly}
Let $K\geq K(q)$ a large enough constant and recall \eqref{defchi}: $$\nu_K(z)=\frac{1}{1+z^{2K}}.$$ Then there holds the pointwise differential bounds:
\bea
\label{differentialcontrolnormbis}
\nonumber&&\frac{d}{ds}\left\|\e\sqrt{\nu_K}\right\|_{L^2_{\rho_r}}^2+\frac K8\left\|\e\sqrt{\nu_K}\right\|_{L^2_{\rho_r}}^2+\left\|\sqrt{\nu_K}\nabla\e\right\|_{L^2_{\rho_r}}^2\\
 & \lesssim & \|\e\|_{H^1_{\rho_Y}}^2+b^{K+\frac 32}+ (\|v\|^2_{L^{\infty}}+b^2)\sum|a_{j,M}|^2,
\eea
\bea
\label{woneqbis}
\nonumber&&\frac{d}{ds}\left(\int |\nabla \e|^{2q+2}\nu_K\rho_rdY\right)+\frac{K}{16q+16}\int |\nabla \e|^{2q+2}\nu_K\rho_rdY\\
&\lesssim & \|\e\|_{L^2_{\rho_Y}}^{2q+2}+\int |\nabla\e|^{2q+2}\rho_YdY+b^{2q+K+\frac 32}+\sum |a_{j,M}|^{2q+2}.
\eea
\end{lemma}

\begin{remark} We more precisely need $K\gtrsim \|\Phi\|_{L^{\infty}}^{p-1}$ in order to absorb the potential terms in the energy estimates below. Also the constants in the rhs of \eqref{differentialcontrolnormbis}, \eqref{woneqbis} do not depend on K.
\end{remark}

\begin{proof}[Proof of Lemma \ref{boundspoly}] This follows from a brute force energy identity using the weight $\frac 1{1+|z|^{2K}}$ to overcome the bounded potential $\Phi^{p-1}$.\\

\noindent{\bf step 1} $L^2$ weighted bound. From \eqref{defequation}:
\bee
\nonumber &&\frac 12\frac{d}{ds}\|\e\sqrt{\nu_K}\|_{L^2_{\rho_r}}^2=(\pa_s\e,\nu_K\e)_{L^2_{\rho_r}}\\
&=&-(\mathcal{L}_Y\e,\nu_K \e)_{L^2_{\rho_r}}+(\Psit_b+L(\e)-\Mod+F(v),\nu_K\e)_{L^2_{\rho_r}}.
\eee
We integrate by parts to compute:
\bee
&&\int\left(-\Delta_Y\e+\frac 12Y\cdot\nabla_Y \e\right)\nu_K\e \rho_rdY\\
&=&\int\left[-\frac{1}{\rho_rr^2}\pa_r(r^2\rho_r\pa_r\e) -\pa_z^2\e+\frac 12z\pa_z\e\right]\nu_K(z)\e r^{2}\rho_rdrdz\\
& = & \int|\nabla_Y\e|^2\nu_K\rho_rdY-\int \e^2\left(\frac{(z\nu_K)'}{4}+\frac{\nu_K''}{2}\right)\rho_rdY
\eee
and hence
\bea
\label{estlineaire}
 -\left(\mathcal{L}_Y\e,\frac{\e}{1+z^{2K}}\right)_{L^2_{\rho_r}} &=& -\int|\nabla_Y\e|^2\nu_K\rho_rdY\\
\nonumber&& +\int \e^2\left(\left(-\frac{2}{p-1}+p\Phi^{p-1}\right)\nu_K+\frac{(z\nu_K)'}{4}+\nu_K''\right)\rho_rdY.
\eea
We now observe that for $|z|\geq z(K)$, 
\be
\label{nceioneoneov}
\frac{(z\nu_K)'}{4}+\frac{\nu_K''}{2}\leq -\frac{K}{4|z|^{2K}}
\ee
and hence for $K\gtrsim 1+\|\Phi\|_{L^{\infty}}$: 
\be
\label{estfondamental}
-(\mathcal{L}_Y\e,\nu_K\e)_{L^2_{\rho_r}}\leq -\int|\nabla_Y\e|^2\nu_KdY-\frac{K}{8}\left\|\e\sqrt{\nu_K}\right\|_{L^2_{\rho_r}}^2+C_K\|\e\|_{L^2_{\rho_Y}}^2,
\ee
where the last term controls the region $|z|\leq z(K)$. The leading order term is estimated from \eqref{eq:decompoistionPsitbinPsitb0andG} \eqref{esterrorinomegabis}:
\bee
|(\e,\nu_K \Psit_b)_{L^2_{\rho_r}}|&\leq& \delta\|\e\sqrt{\nu_K}\|_{L^2_{\rho_r}}^2+c_\delta\int_{|Z|\leq 2\delta} \frac{1}{1+|z|^{2K}}\left[b^{2n+2}+b^2|Z|^{4n+4}\right]dz\\
&&+ c_\de b^2\int_{|Z|\geq \de}\frac{dz}{1+|z|^{2K}}.
\eee
We estimate after changing variables $Z=z\sqrt{b}$:
\bee
\int_{|Z|\leq 2\delta} \frac{1}{1+|z|^{2K}}\left[b^{2n+2}+b^2|Z|^{4n+4}\right]\rho_rdz&\lesssim &b^{2n+2}+\int_{|Z|\leq 2\delta}\frac{b^2}{\sqrt{b}}b^{K}\frac{|Z|^{4n+4}}{|Z|^{2K}}dZ\\
&& +b^2\int_{|Z|\geq \delta}\frac{1}{\sqrt{b}}b^{K}\frac{1}{|Z|^{2K}}dZ\\
& \lesssim & b^{K+\frac 32}
\eee
provided $n\geq n(K)$ has been chosen large enough in Lemma \ref{lemmaapproximate}. We next integrate by parts like for the proof of \eqref{integrationbyparts} to compute:
\be
\label{Lterm}
|(\nu_K \e,\Lambda_Y\e)_{L^2_{\rho_r}}|\lesssim \int \nu_K(1+r^2)\e^2\rho_rr^2drdz\lesssim \int \nu_K (|\nabla\e|^2+\e^2)\rho_rr^2drdz
\ee
 where we used \eqref{estimportante} in the last step, and hence from \eqref{roughboundparameters}:
 $$|(L(\e),\nu_K \e|\lesssim b\left(\|\nabla\e\sqrt{\nu_K}\|_{L^2_{\rho_r}}^2+\|\e\sqrt{\nu_K}\|_{L^2_{\rho_r}}^2\right).$$ 
 To estimate the modulation equation terms, we first observe from \eqref{cnecneoneo} that 
 \be
 \label{cnekoneknvoe}
 \|\sqrt{\nu_K}\phi_{j,2M}\|_{L^2_{\rho_r}}+\|\Lambda\Phit_b\sqrt{\nu_K}\|_{L^2_{\rho_r}}+\|\pa_b\Phit_b\sqrt{\nu_K}\|_{L^2_{\rho_r}}\lesssim 1, \ \ -\ell_0\leq j\leq-1,\, 0\leq M\leq M(j) 
 \ee provided $K$ satisfies \eqref{eq:necessaryconditionK} and hence from \eqref{controlMOd}:
\bea
\label{boudnmodualiton}
\nonumber \|\Mod\sqrt{\nu_K}\|_{L^2_{\rho_r}} &\lesssim& \sum|(a_{j,M})_s-(\l_j+M)a_{j,M}|\\
\nonumber &+& \left|\lsl+\frac 12\right|\sum|a_{j,M}|+\left|\lsl+\frac 12-M(b)\right|+|b_s+bB(b)|\\
& \lesssim & b^{n+1}+b\left(\|\e\|_{L^2_{\rho_Y}}+\sum |a_{j,M}|\right)
\eea
which yields the bound:
$$|(\Mod,\nu_K\e)_{L^2_{\rho_r}}|\leq b^{2n+1}+b\|\e\|^2_{L^2_{\rho_Y}}+ b^2\sum |a_{j,M}|^2+\|\e\sqrt{\nu_K}\|^2_{L^2_{\rho_r}}.$$
The small linear term is estimated in brute force using $\|\Phit_b\|_{L^{\infty}}+\|\Phi\|_{L^{\infty}}\lesssim 1$:
$$|(F_1(v),\nu_K\e)_{L^2_{\rho_r}}|\lesssim \int (|\e|+|\psi|)\nu_K|\e|e^{-\frac{r^2}{2}}r^2dz\leq \frac{K}{20}\|\e\sqrt{\nu_K}\|_{L^2_{\rho_r}}^2+\sum |a_{j,M}|^2$$ 
where we used \eqref{cnekoneknvoe} in the last step. The nonlinear term is estimated as before:
\bee
|(F_2(v),\nu_K\e)_{L^2_{\rho_r}}|&\lesssim & \int \nu_K|\e|v^2\rho_rr^2drdz \leq\frac{K}{20}\int\nu_K|\e|^2\rho_rr^2drdz+\int \nu_K|v|^4\rho_rr^2drdz \\
& \leq & \frac{K}{20}\int\nu_K|\e|^2\rho_rr^2drdz+\|v\|^2_{L^{\infty}}\int\nu_K |v|^2\rho_rr^2drdz\\
& \leq&\frac{K}{20}\int\nu_K|\e|^2\nu_K\rho_rr^2drdz+C\|v\|^2_{L^{\infty}}\sum|a_{j,M}|^2,
\eee
where we used \eqref{cnekoneknvoe} in the last step, and \eqref{differentialcontrolnormbis} follows.\\

\noindent{\bf step 2} $\dot{W}^{1,2q+2}$ weighted bound. Let $\e_i=\pa_i\e$, for $i=1, 2, 3, 4$. We  compute from \eqref{eqei}:
\bee
&&\frac{1}{2q+2}\frac{d}{ds}\int \nu_K \e_i^{2q+2}\rho_rdY=\int \nu_K\e_i^{2q+1}\pa_s\e_i\\
&=& -\left((\L+1)\e_i, \nu_K\e_i^{2q+1}\right)_{L^2_{\rho_r}}+\left(\e_i^{2q+1},\nu_K\pa_i\left[\Psit_b-\Mod+L(\e)+F(v)\right]\right)_{L^2_{\rho_r}}\\
&& +\left(\e_i^{2q+1},\nu_Kp(p-1)\Phi^{p-2}\pr_i\Phi \ep\right)_{L^2_{\rho_r}}.
\eee
 
We integrate by parts to compute:
\bee
&&\int\left(-\Delta_Y\e_i+\frac 12Y\cdot\nabla_Y \e_i\right)\nu_K\e^{2q+1} \rho_rdY\\
&=&\int\left[-\frac{1}{\rho_rr^2}\pa_r(r^{2}\rho_r\pa_r\e_i)-\pa_z^2\e_i+\frac 12z\pa_z\e_i\right]\nu_K(z)\e_i^{2q+1} r^{2}\rho_rdrdz\\
& = & (2q+1)\int \e_i^{2q}|\pa_r\e_i|^2\nu_K\rho_rr^{2}drdz+(2q+1)\int (\pa_z\e_i)^2\nu_K\rho_r r^2drdz\\
&-& \frac{1}{2q+2}\int\e_i^{2q+2}\left(\frac{(z\nu_K)'}{2}+\nu_K''\right)\rho_rr^{2}drdz\\
& \geq & c\int|\nabla_Y(\e_i^{q+1})|^2\nu_K\rho_rdY+\frac{K}{8q+8}\int \e_i^{2q+2}\nu_K \rho_rdY-C_K\int \e_i^{2q+2}\rho_YdY
\eee
where we used \eqref{nceioneoneov} in the last step, and hence
\bee
 -\left((\L+1)\e_i, \nu_K\e_i^{2q+1}\right)_{L^2_{\rho_r}}&\leq& -\left\{c\int|\nabla_Y(\e_i^{q+1})|^2\nu_K\rho_rdY+\frac{K}{16q+16}\int \e_i^{2q+2}\nu_K \rho_rdY\right\}\\
 &+& C_K\int \e_i^{2q+2}\rho_YdY.
 \eee
The leading order term is estimated from \eqref{eq:decompoistionPsitbinPsitb0andG} \eqref{esterrorinomegabis}:
\bee
\left|(\e_i^{2q+1},\nu_K \pa_i\Psit_b)_{L^2_{\rho_r}}\right|&\lesssim& \int \nu_K\left(|\e_i|^{2q+2}+|\pa_i\Psit_b|^{2q+2}\right)\rho_rdY\\
&\lesssim& \int \nu_K|\e_i|^{2q+2}\rho_rdY
+\int_{|Z|\leq 2\delta} \frac{1}{1+|z|^{2K}}\left[b^{2n+2}+b|Z|^{4n+4}\right]^{2q+2}dz\\
&& +b^{2q+2}\int_{|Z|\geq \delta} \frac{1}{1+|z|^{2K}}dz\\
& \lesssim & \int \nu_K|\e_i|^{2q+2}\rho_rdY+b^{2q+K+\frac 32}.
\eee
We next integrate by parts and use \eqref{estimportante} to estimate:
$$
|(\nu_K \e_i^{2q+1},\pa_i\Lambda_Y\e)_{L^2_{\rho_r}}|\lesssim  \int \nu_K(1+r^2)\e_i^{2q+2}\rho_rdY\lesssim \int \nu_K \left[\e_i^{2q+2}+|\nabla_Y(\e_i^{q+1})|^2\right]\rho_rdY
$$
and hence from \eqref{roughboundparameters}:
$$|(\nu_K \e_i^{2q+1},\pa_iL(\e))_{L^2_{\rho_r}}|\lesssim b\int \nu_K \left[\e_i^{2q+2}+|\nabla_Y(\e_i^{q+1})|^2\right]\rho_rdY.$$ 
The modulation equation terms are estimated in brute force for $K\geq K(q)$ large enough from \eqref{controlMOd}:
\bee
|(\nu_K \e_i^{2q+1},\pa_i\Mod)_{L^2_{\rho_r}}|&\lesssim & \int \nu_K\e_i^{2q+2}\rho_rdY+\int |\pa_i\Mod|^{2q+2}\nu_K\rho_rdY\\&\lesssim & \int \nu_K\e_i^{2q+2}\rho_rdY+b^{2q+2}\left[\|\e\|_{L^2_{\rho_Y}}^{2q+2}+\sum |a_{j,M}|^{2q+2}\right]+b^{2q+K+\frac 32}.
\eee
Also, we have
 \bee
\left|\left(\e_i^{2q+1},\nu_Kp(p-1)\Phi^{p-2}\pr_i\Phi \ep\right)_{L^2_{\rho_r}}\right| \lesssim  \int |\e_i|^{2q+1}|\e|\nu_K\rho_rdY\lesssim  \int |\e_i|^{2q+2}\nu_K\rho_rdY.
 \eee 
For the nonlinear term, we estimate in brute force from \eqref{pointwisenls2}:
$$
|(\nu_K \e_i^{2q+1},\pa_iF(v))_{L^2_{\rho_r}}|\lesssim \int (|v|+|\nabla_Yv|)\nu_K |\e_i|^{2q+1}\rho_rdY\lesssim \int |\e_i|^{2q+2}\nu_K\rho_rdY+\sum |a_{j,M}|^{2q+2}.$$
The collection of above bounds concludes the proof of \eqref{woneqbis}.
\end{proof}


\subsection{Outer global $W^{1,2q+2}$ bound}


We recall $$v=\e+\psi$$ and now aim at propagating an {\it unweighted global} $W^{1,2q+2}$ decay estimate for $v$. We rewrite \eqref{vequationbizarre} as
\be
\label{vequationbis}
\pa_sv-\Delta_Y v -\lsl \Lambda_Y v=h, 
\ee
with $$h=\Psit_b+\left(\lsl+\frac 12-M(b)\right)\Lambda \Phit_b-(b_s+bB(b))\pa_b\Phit_b+\widehat{F}(v), \ \ \widehat{F}=(\Phit_b+v)^p-\Phit_b^p.$$
 
\begin{lemma}[Global $W^{1,2q+2}$ bound]
\label{lemmehtwoweight}
There holds the Lyapunov type monotonicity formula
\bea
\label{decaylp}
\frac{d}{ds}\left\{\int |v|^{2q+2}dY\right\}+c\int |v|^{2q+2}dY\lesssim b^{2q+\frac 32}+\frac{1}{b^K}\int \frac{|v|^{2q+2}}{1+|z|^{2K}}\rho_rdY,
\eea
\bea
\label{decaylpbis}
&&\frac{d}{ds}\left\{\int |\nabla_Yv|^{2q+2}dY\right\}+c \int|\nabla_Yv|^{2q+2}dY\\
\nonumber &\lesssim & b^{2q+\frac 32}+\frac{1}{b^K}\int \frac{|\nabla_Yv|^{2q+2}+|v|^{2q+2}}{1+|z|^{2K}}\rho_rdY,
\eea
for some universal constant $c(q)>0$.
\end{lemma}

\begin{proof}[Proof of Lemma \ref{lemmehtwoweight}] {\bf step 1} $L^{2q+2}$ bound. We compute from \eqref{vequationbis}:
$$\frac1{2q+2}\frac d{ds}\int v^{2q+2}dY=\int v^{2q+1}\pa_sv =\int v^{2q+1}\left[\Delta_Y v+\lsl\Lambda v+h\right]dY.$$
The linear term is computed by integration by parts:
\bea
\label{nceoineone}
&&\int v^{2q+1}\left(\Delta_Y v+\lsl\Lambda v\right)dY\\
\nonumber &  = & -(2q+1)\int v^{2q}|\nabla_Yv|^2dY+\lsl\left(\frac{2}{p-1}-\frac{4}{2q+2}\right)\int  v^{2q+2}dY.
\eea
Observe using \eqref{roughboundparameters} that
\bee
\lsl\left[\frac{2}{p-1}-\frac{4}{2q+2}\right]\int  v^{2q+2}dY &=& -\left(\frac 12+O(b)\right)\left(\frac{2}{p-1}-\frac{4}{2q+2}\right)\int  v^{2q+2}dY\\
&\leq& -c\int  v^{2q+2}dY
\eee
where $c>0$ for $b$ small enough and $q$ large enough. Next by H\"older:
$$\int v^{2q+1}hdY\leq \delta v^{2q+2}dY+c_\delta\int h^{2q+2}dY$$ and we now estimate the $h$ terms. First from \eqref{eq:decompoistionPsitbinPsitb0andG} \eqref{esterrorinomegabis}:
\bee
\int |\Psit_b|^{2q+2}&\lesssim&\int_{|Z|\leq 2\delta} \left|\frac{b^{n+1}+b|Z|^{2n+2}}{\la r\ra^{\frac{2}{p-1}-\frac{1}{n}}}\right|^{2q+2}dY\\
&&+b^{2q+2}\int_{|Z|\geq \delta} (|\pr^2_ZG|+|Z\pr_ZG|)^{2q+2}dY\\
&\lesssim& b^{2q+2-\frac 12}
\eee
for $q$ large enough, where we used in the last inequality the fact that in view of \eqref{boundphir} \eqref{cnekoneoneo} \eqref{formulag}, we have
$$|\pr^2_ZG|+|Z\pr_ZG| \lesssim \frac{1}{\left(1+r^2+|Z|^2\right)^{\frac{1}{p-1}}}.$$
In order to treat the modulation equation terms, we compute
\bee
|\Lambda_Y \Phi_b|=\left|\frac{1}{\mu^{\frac{2}{p-1}}}\left(1-\frac{Z^2}{1+Z^2}\right)\Lambda \Phi\left(\frac{r}{\mu}\right)\right|\lesssim \frac{1}{\mu^{\frac 2{p-1}}+\la r\ra^{\frac{2}{p-1}}}\lesssim \frac{1}{(1+r^2+bz^2)^{\frac{1}{p-1}}}
\eee 
and hence for $q$ large enough using \eqref{decayv}:
$$ \int |\Lambda_Y \Phit_b|^{2q+2}\lesssim \frac{1}{\sqrt{b}}.$$
Similarly:
$$|\pa_b\Phi_b|\lesssim \frac{1}{2b}|Z\pa_ZG|\lesssim \frac{1}{b}\frac{Z^2}{1+Z^2}\frac{1}{\mu^{\frac 2{p-1}}}\Lambda\Phi\left(\frac{r}{\mu}\right)\lesssim \frac{1}{b}\frac{1}{(1+r^2+bz^2)^{\frac{1}{p-1}}}$$ and hence 
$$\int |\pa_b \Phit_b|^{2q+2}dY\lesssim \frac{1}{\sqrt{b}b^{2q+2}}.$$ We conclude using \eqref{controlMOd}:
\bee
&&\int \left|\left(\lsl+\frac 12-M(b)\right)\Lambda \Phit_b-(b_s+bB(b))\pa_b\Phit_b\right|^{2q+2}dY\\
&\lesssim& \frac{1}{\sqrt{b}}\left(b^n+\|\e\|_{L^2_{\rho_Y}}+\sum|a_{j,M}|\right)^{2q+2}\lesssim b^{\frac n2}
\eee
where we used in the last inequality the bounds \eqref{bparmebis} \eqref{controlunstable} \eqref{poitwiseboundhtwo}.

We now turn to the nonlinear term whose control relies on the polynomially weighted bounds of Lemma \ref{boundspoly}. Indeed, we estimate by homogeneity $$|\widehat{F}(v)|\lesssim |\Phit_b||v|+\|v\|_{L^{\infty}}|v|$$ from which:
$$
\int |\widehat{F}(v)|^{2q+2}dY\lesssim \int |\Phit_b|^{2q+2}|v|^{2q+2}+\|v\|_{L^\infty}^{2q+2}\int |v|^{2q+2}dY.$$ The second term is treated thanks to $\|v\|_{L^{\infty}}\ll 1$, and we split the first term using 
\be
\label{cneokcneoneon}
\|\pa_i^k\Phit_b\|_{L^{\infty}(|Z|\geq A)}+\|\pa_i^k\Phit_b\|_{L^{\infty}(r\geq A)}<\delta\ll1, \ \ k=0,1
\ee for $A$ large enough, and hence
\bee
&&\int |\Phit_b|^{2q+2}|v|^{2q+2}\\
& \lesssim & \int_{|Z|\geq A} |\Phit_b|^{2q+2}|v|^{2q+2}+\int_{r\geq A} |\Phit_b|^{2q+2}|v|^{2q+2}+\int_{|Z|\leq A,r\leq A}|\Phit_b|^{2q+2}|v|^{2q+2}\\
& \lesssim&  \delta\int |v|^{2q+2}dY+\frac{C(A)}{b^K}\int \frac{|v|^{2q+2}}{1+|z|^{2K}}\rho_rdY.
\eee
The collection of above bounds concludes the proof of \eqref{decaylp}.\\

\noindent{\bf step 2} $\dot{W}^{1,2q+2}$ bound. Let $v_i=\pa_iv$ for $i=1, 2, 3, 4$. Then from \eqref{vequationbis}:
$$\pa_sv_i-\Delta_Y v_i -\lsl \left[v_i+\Lambda_Y v_i\right]=\pa_ih,$$ and hence: 
$$\frac1{2q+2}\frac d{ds}\int v_i^{2q+2}dY=\int v_i^{2q+1}\pa_sv_i =\int v_i^{2q+1}\left[\Delta_Y v_i+\lsl(v_i+\Lambda_Y v_i)+\pa_ih\right]dY.$$ The linear term is computed from \eqref{nceoineone}:
\bee
&&\int v_i^{2q+1}\left[\Delta_Y v_i+\lsl(v_i+\Lambda_Y v_i)\right]dY\\
&=& -(2q+1)\int v_i^{2q}|\nabla_Yv_i|^2dY+\lsl\left(1+\frac{2}{p-1}-\frac{4}{2q+2}\right)\int  v_i^{2q+2}dY\\
&=& -\left(\frac 12+O(b)\right)\left(1+\frac{2}{p-1}-\frac{4}{2q+2}\right)\int  v_i^{2q+2}dY\\
&\leq& -c\int  v_i^{2q+2}dY.
\eee
Next, we have by H\"older:
$$\int v_i^{2q+1}\pa_ihdY\leq \delta v_i^{2q+2}dY+c_\delta\int (\pa_ih)^{2q+2}dY$$ 
and we now estimate the $\pr_ih$ terms.
From \eqref{esterrorinomegabis}:
\bee
\int |\pa_i\Psit_b|^{2q+2}&\lesssim&\int_{|Z|\leq 2\delta} \left|\frac{b^{n+1}+b|Z|^{2n+2-1}}{\la r\ra^{\frac{2}{p-1}-\frac{1}{n}}}\right|^{2q+2}dY\\
&&+b^{2q+2}\int_{|Z|\geq \delta} (|\pr_i\pr^2_ZG|+|\pr_i(Z\pr_ZG)|)^{2q+2}dY\\
&\lesssim& b^{2q+2-\frac 12}
\eee
for $q$ large enough. For the modulation equation terms, we estimate in brute force as above
$$ \int |\pa_i\Lambda_Y \Phit_b|^{2q+2}\lesssim \frac{1}{\sqrt{b}}, \ \ \int |\pa_i\pa_b \Phit_b|^{2q+2}dY\lesssim \frac{1}{\sqrt{b}b^{2q+2}}$$ from which using \eqref{controlMOd}:
$$\int \left|\left(\lsl+\frac 12-M(b)\right)\pa_i\Lambda \Phit_b-(b_s+bB(b))\pa_i\pa_b\Phit_b\right|^{2q+2}dY\lesssim b^{\frac n2}.$$
We now estimate the nonlinear term by homogeneity:
\bee
|\pa_i\widehat{F}(v)|\lesssim |\pa_i\Phit_b||v|+|\pa_iv|\left[\|v\|_{L^{\infty}}^{p-1}+|\Phit_b|^{p-1}\right]
\eee from which for $A$ large enough using \eqref{cneokcneoneon}:
\bee
\int |\pa_i\widehat{F}(v)|^{2q+2}dY&\lesssim& \delta \int (|v|^{2q+2}+|v_i|^{2q+2})dY+\int_{|Z|\leq A,r\leq A}(|v|^{2q+2}+|v_i|^{2q+2})dY\\
& \lesssim &  \delta \int (|v|^{2q+2}+|v_i|^{2q+2})dY+\frac{C(A)}{b^K}\int \frac{|v|^{2q+2}+|v_i|^{2q+2}}{1+|z|^{2K}}\rho_rdY.
\eee
The collection of above bounds concludes the proof of \eqref{decaylpbis}.
\end{proof}


\subsection{Conclusion}\label{sec:conclusionproofpropboot}


We are now in position to conclude the proof of Proposition \ref{bootstrap} which then easily implies Theorem \ref{thmmain}.\\

\begin{proof}[Proof of Proposition \ref{bootstrap}] We recall that we are arguing by contradiction assuming \eqref{assumptioncontradiction}. We first  show that the bounds \eqref{controlsclaing}, \eqref{bparmebis}, \eqref{poitwiseboundhtwo}, \eqref{neneoneovneoneov}, \eqref{cekneoneoncoeno}, \eqref{controlsobolev} can be improved on $[s_0,s^*]$, and then, the existence of the data $a_{j,M}(0)$ follows from a classical topological argument \`a la Brouwer.\\

\noindent{\bf step 1} Improved control of the geometrical parameters. We estimate from \eqref{roughboundparameters}:
$$\frac12(1-\delta)<-\lsl<\frac 12(1+\delta), \ \ 0<\delta\ll1$$ which implies $$(\l(s)e^{\frac s4})'<0, \ \ (\l(s)e^{2s})'>0$$
and hence using \eqref{scalingsmall} $$0<\l(s)<\l(s_0)e^{\frac{s_0}{4}}e^{-\frac{s}{4}}<\frac 12 e^{-\frac{s}{4}},$$ and \eqref{controlsclaing} is improved. For the $b$ parameter, we estimate from \eqref{controlMOd}, \eqref{bparmebis}, \eqref{controlunstable}, \eqref{poitwiseboundhtwo} and \eqref{calculloi}:
$$|b_s+c_1b^2|\lesssim \frac{1}{s^3}, \ \ c_1>0.$$ Hence $$\left|\frac{d}{ds}\left(-\frac 1b\right)+c_1\right|\lesssim \frac{1}{s^3b^2}\lesssim \frac{1}{s}$$ which time integration using \eqref{bparme} yields:
$$\left|-\frac{1}{b(s)}+c_1s\right|\leq \sqrt{s}\ \ \mbox{and thus}\ \ \frac{1}{2c_1 s}<b(s)<\frac{2}{c_1s}$$ which improves \eqref{bparmebis}.\\

\noindent{\bf step 2} Improved Sobolev bounds. We now systematically integrate in time the monotonicity formulas of Lemmas \ref{boundsexpo}, \ref{boundspoly}, \ref{lemmehtwoweight} to improve the bounds \eqref{poitwiseboundhtwo}-\eqref{controlsobolev}.\\

\noindent{\em Exponential norms}. We rewrite \eqref{differentialcontrolnorm} using \eqref{bparmebis}, \eqref{controlunstable}, \eqref{poitwiseboundhtwo}, \eqref{poitwisebound} and obtain:
$$
\frac{d}{ds}\|\e\|_{L^2_{\rho_Y}}^2+c\|\e\|_{L^2_{\rho_Y}}^2 \lesssim \frac{1}{s^{n+\delta_q}}
$$
which time integration with \eqref{poitwiseboundhtwoinitial} yields:
\bea
\label{ltwoboot}
\nonumber \|\e(s)\|_{L^2_{\rho_Y}}^2&\leq& \|\e(0)\|_{L^2_{\rho_Y}}^2e^{-c(s-s_0)}+e^{-s}\int_{s_0}^{s}\frac{e^{c\sigma}}{\sigma^{n+\delta_q}}d\sigma\lesssim  \frac{1}{s^{2n}}\left(\frac{s}{s_0}\right)^{2n}e^{-c(s-s_0)}+\frac{1}{s^{n+\delta_q}}\\
& \lesssim & \frac{1}{s^{n+\delta_q}}
\eea
where we used 
$$
\left(\frac{s}{s_0}\right)^{2n}e^{-c(s-s_0)}\leq 1\ \ \mbox{for}\ \ s\geq s_0
$$
since $s\mapsto \left(\frac{s}{s_0}\right)^{2n}e^{-c(s-s_0)}$ is non increasing on $[s_0,+\infty)$ for $s_0(n)$ large enough. 
We similarly rewrite \eqref{differentialcontrolnormbisbis}:
$$\frac{d}{ds}\|\L_Y\e\|_{L^2_{\rho_Y}}^2+c\|\L_Y\e\|_{L^2_{\rho_Y}}^2 \lesssim \frac{1}{s^{n+\delta_q}}$$ which similarly yields with \eqref{poitwiseboundhtwoinitial}: 
\be
\label{ltwobootbis}
\|\L_Y\e(s)\|_{L^2_{\rho_Y}}^2\lesssim \frac{1}{s^{n+\delta_q}}.
\ee
We now recall $$(\mathcal L_Y\e,\e)_\rho=\|\nabla \e\|^2_{L^2_{\rho_Y}}+\int\left(\frac 2{p-1}-p\Phi^{p-1}\right)|\e|^2\rho_Y dY$$ which together with \eqref{ltwoboot}, \eqref{ltwobootbis} implies: 
$$
\|\nabla \e\|_{L^2_{\rho_Y}}^2\leq (\mathcal L_Y\e,\e)_{L^2_{\rho_Y}}+C\|\e\|_{L^2_{\rho_Y}}^2\lesssim  \frac{1}{s^{n+\delta_q}}.
$$
This implies from \eqref{estimationapoinds}:
$$\|\Delta \e\|_{L^2_{\rho_Y}}^2\lesssim  \|\L_Y\e(s)\|_{L^2_{\rho_Y}}^2+\|\e\|_{H^1_{\rho_Y}}^2\lesssim \frac{1}{s^{n+\delta_q}}$$ which together with \eqref{ltwoboot}, \eqref{ltwobootbis} yields 
\be
\label{cneineonoec}
\|\e\|_{H^2_{\rho_Y}}^2\lesssim \frac{1}{s^{n+\delta_q}}<\frac 1{2s^{n}}
\ee
which improves \eqref{poitwiseboundhtwo}. Similarly from \eqref{woneqlocbound}, \eqref{bparmebis}, \eqref{controlunstable}, \eqref{poitwiseboundhtwo}:
$$\frac{d}{ds}\|\nabla\e\|^{2q+2}_{L^{2q+2}_{\rho_Y}}+c\|\nabla\e\|^{2q+2}_{L^{2q+2}_{\rho_Y}}\lesssim \frac{1}{s^{(q+1)(n+1)}}$$ which time integration using \eqref{poitwiseboundhtwoinitial} ensures: $$\|\nabla\e\|^{2q+2}_{L^{2q+2}_{\rho_Y}}\lesssim \frac{1}{s^{(q+1)(n+1)}}\leq \frac{1}{2s^{n(q+1)}}$$ and \eqref{neneoneovneoneov} is improved.\\

\noindent{\em Polynomial norms}.  We rewrite \eqref{differentialcontrolnormbis} using \eqref{bparmebis}, \eqref{controlunstable}, \eqref{poitwiseboundhtwo}, \eqref{controlsobolev} as:
\bee
\nonumber&&\frac{d}{ds}\left\|\e\sqrt{\nu_K}\right\|_{L^2_{\rho_r}}^2+\frac{K}{8}\left\|\e\sqrt{\nu_K}\right\|_{L^2_{\rho_r}}^2\lesssim \frac{1}{s^{K+\frac 32}}
\eee
for $n\geq n(K)$ large enough, which time integration using  \eqref{cekneoneoncoenoinitial} yields:
$$
\left\|\e(s)\sqrt{\nu_K}\right\|_{L^2_{\rho_r}}^2 \lesssim\frac{1}{s^{K+\frac 32}}<\frac{1}{2s^{K+1}}.$$
Similarly from \eqref{woneqbis}:
$$\frac{d}{ds}\left(\int |\nabla \e|^{2q+2}\nu_K\rho_rdY\right)+\frac{K}{16q+16}\int |\nabla \e|^{2q+2}\nu_K\rho_rdY\lesssim \frac{1}{s^{K+2q+\frac 32}}$$ which together with \eqref{cekneoneoncoenoinitial} ensures: $$\int |\nabla \e(s)|^{2q+2}\nu_K\rho_rdY\lesssim \frac{1}{s^{K+2q+\frac 32}}< \frac{1}{2s^{K+2q+1}}.$$
This yields an improvement of \eqref{cekneoneoncoeno}.\\

\noindent{\em Global norms}. We use the lossy bound
\bee
\int\frac{ |v|^{2q+2}}{1+|z|^{2K}}\rho_rdY&\lesssim&  \|v\|_{L^\infty}^{2q}\int \frac{|v|^{2}}{1+|z|^{2K}}\rho_rdY\lesssim\int\frac{ |\e|^{2}}{1+|z|^{2K}}\rho_rdY+\sum|a_{j,M}|^2\\
& \lesssim &  \frac{1}{s^{K+1}}
\eee
which injected into \eqref{decaylp}, \eqref{decaylpbis} yields
$$
\frac{d}{ds}\|v\|_{W^{1,2q+2}}^{2q+2}+c\|v\|_{W^{1,2q+2}}^{2q+2}\lesssim \frac{1}{s^{2q+\frac 32}}+\frac{1}{b^K}\left[\frac{1}{s^{K+1}}\right]\lesssim \frac 1s.
$$
Integrating in time using \eqref{controlsobolevinitial} ensures $$\|v(s)\|_{W^{1,2q+2}}^{2q+2}\lesssim \frac{1}{s}$$ which improves \eqref{controlsobolev} provided $\delta_q$ has been chosen small enough.\\

\noindent{\bf step 3} Brouwer fixed point argument. In view of the above improvements of \eqref{controlsclaing}, \eqref{bparmebis}, \eqref{poitwiseboundhtwo}, \eqref{neneoneovneoneov}, \eqref{cekneoneoncoeno}, \eqref{controlsobolev}, we conclude from an elementary continuity argument that \eqref{assumptioncontradiction} implies the exit condition:
\be
\label{exitcondition}
\sum_{j=-\ell_0}^{-2}\sum_{M=0}^{M(j)}(a_{j,M}(s^*))^2=\frac{1}{(s^*)^{n}}.
\ee
On the other hand, we estimate from \eqref{controlMOd}, \eqref{poitwiseboundhtwo}:
$$\sum|(a_{j,M})_s+(\l_j+M)a_{j,M}|\lesssim \frac{1}{s^{\frac n2+1}}.$$
Also, from the non degeneracy \eqref{hypspectral}, there exists $c>0$ such that
$$\l_j+M\leq -c<0,\,\,\,\, -\ell_0\leq j\leq -2, \, 0\leq M\leq M(j)$$
and hence
\bee
&&\frac{d}{ds}\sum s^n(a_{j,M}(s))^2=s^n\sum a_{j,M}\left[(a_{j,M})_s+\frac{n}{s}a_{j,M}\right]\\
& = & s^n\sum a_{j,M}\left[(a_{j,M})_s+(\l_j+M)a_{j,M}+\frac{n}{s}a_{j,M}\right]-s^n\sum (\l_j+M)|a_{j,M}|^2\\
& \geq & cs^n\sum |a_{j,M}|^2+O\left(\frac{s^n}{s^{n+1}}\right)
\eee
which implies from \eqref{exitcondition} the outgoing flux condition: $$\frac{d}{ds}\left\{\sum_{j=-\ell_0}^{-2}\sum_{M=0}^{M(j)} s^n(a_{j,M}(s))^2\right\}_{|s=s^*}>\frac{c}{2}>0.$$ 
We conclude from standard argument that the map
$$(a_{j,M}(0)s_0^n)_{  -\ell_0\leq j\leq -2, \, 0\leq M\leq M(j)}\to (a_{j,M}(s_*)s_*^n)_{  -\ell_0\leq j\leq -2, \, 0\leq M\leq M(j)}$$
is continuous on the unit ball of $\mathbb{R}^N$ where
$$N=\sum_{j=-\ell_0}^{-2}(1+M(j))$$
and the identity on its boundary a contradiction to Brouwer's Theorem. This concludes the proof of Proposition \ref{bootstrap}.
\end{proof}

We are now in position to conclude the proof of Theorem \ref{thmmain}.

\begin{proof}[Proof of Theorem \ref{thmmain}] Let an initial data as in Proposition \ref{bootstrap}. The $s$ time being global, the integration of the modulation equations \eqref{controlMOd} with the bounds \eqref{poitwiseboundhtwo} easily leads to the laws $$b(s)=\frac{1}{c_1s}+O\left(\frac 1{s^2}\right)$$ and $$\lsl=-\frac 12+M(b)+O\left(\frac 1{s^2}\right)=-\frac 12+\frac{1}{s}+O\left(\frac{1}{s^2}\right)$$
where we used the fact that $d_1=1$ in the last equality in view of \eqref{calculloi}. Hence $$\l(s)=e^{-\frac{s}{2}+O(\log s)}=e^{-\frac{s}{2}}\left[1+O\left(\frac{\log s}{s}\right)\right].$$ This implies that the life time of the solution is finite$$ T=\int_{s_0}^{+\infty}\l^2(s)ds<+\infty$$ and the blow up is self similar $$T-t=\int_{s}^{+\infty}\l^2(s)ds=  e^{-s}\left[1+O\left(\frac{\log s}{s}\right)\right], \ \ \l(t)=\sqrt{T-t}(1+o(1)).$$ The fact that the above construction defines a Lipschitz manifold of initial data in the $W^{1,2q+2}\cap H^2$ topology is now classical, see \cite{CRS}, and the details are left to the reader. This concludes the proof of Theorem \ref{thmmain}.
\end{proof}

\begin{appendix}



\section{Coercivity estimates}
\label{appendcoerc}

\begin{lemma}[Exponential Hardy]
Let $\nu(z)\geq 0$ and $u(r,z)$ with cylindrical symmetry, then:
\be
\label{estimportante}
\int \nu|u|^{2}(1+r^2)e^{-\frac{r^2}{4}}r^2drdz\lesssim \int \nu(|u|^{2}+|\nabla_Yu|^{2})e^{-\frac{r^2}{4}}r^2drdz.
\ee
Moreover:
\be
\label{estimationapoinds}
\|\Delta_Y u\|_{L^2_{\rho_Y}}^2\lesssim \left\|-\Delta u+\frac 12Y\cdot\nabla u\right\|_{L^2_{\rho_Y}}^2+\|u\|_{H^1_{\rho_Y}}^2.
\ee
\end{lemma}

\begin{proof} {\bf step 1} Proof of \eqref{estimportante}. By density, we assume $u\in \matchal D(\Bbb R^4)$. We use $\pa_r\rho_r=-r\rho_r/2$ and integrate by parts to compute:
$$
 \int_0^{+\infty} \e^{2}r^2e^{-\frac{r^2}{4}}r^2dr=\int_0^{+\infty} 6\e^{2}e^{-\frac{r^2}{4}}r^2dr+4\int_0^{+\infty}\pa_r\e\e e^{-\frac{r^2}{4}}r r^2dr
 $$
 and hence from H\"older:
 \bee
 &&\int_0^{+\infty} \e^{2}r^2e^{-\frac{r^2}{4}}r^2dr\lesssim \int_0^{+\infty}|\pa_r\e||r\e| e^{-\frac{r^2}{4}} r^2dr +\int_0^{+\infty} \e^{2}e^{-\frac{r^2}{4}}r^2dr
 \\
 & \leq & c_\delta \int_0^{+\infty} |\pa_r\e|^{2}r^2e^{-\frac{r^2}{4}}r^2dr+\delta\int_0^{+\infty} \e^{2}r^{2}e^{-\frac{r^2}{4}}r^2dr +\int_0^{+\infty} \e^{2}e^{-\frac{r^2}{4}}r^2dr.
\eee
Hence
$$\int_0^{+\infty} \e^{2}r^2e^{-\frac{r^2}{4}}r^2dr\lesssim \int_0^{+\infty} \left[|\pa_r\e|^{2}+\e^{2}\right]r^2e^{-\frac{r^2}{4}}dr.$$
We now multiply by $\nu$ and integrate in $z$, and \eqref{estimportante} is proved.\\

\noindent{\bf step 2}. Proof of \eqref{estimationapoinds}. We compute:
$$\left\|-\Delta u +\frac 12Y\cdot\nabla u\right\|_{L^2_{\rho_Y}}^2=\|\Delta u\|^2_{L^2_{\rho_Y}}+\frac 14\|Y\cdot\nabla u\|_{L^2_{\rho_Y}}^2-\int(\Delta u) Y\cdot\nabla u\rho_Y dY.$$
To compute the crossed term, let $u_\l(Y)=u(\l Y),$ then $$\int|\nabla u_\l(Y)|^2\rho_Y dY=\frac{1}{\l^2}\int|\nabla u(Y)|^2\rho_Y\left(\frac{Y}{\l}\right)dy$$ and hence differentiating in $\l$ and evaluating at $\l=1$:
$$2\int \nabla u\cdot\nabla(Y\cdot\nabla u)\rho_Y dY=\int|\nabla u|^2(-2\rho_Y-Y\cdot\nabla \rho_Y)dy$$
i.e. 
$$2\int Y\cdot\nabla u(\rho_Y\Delta u+\nabla u\cdot\nabla \rho_Y)=2\int|\nabla u|^2\left(\rho_Y +\frac{1}{2}Y\cdot\nabla \rho_Y\right)dy$$
which using $\nabla \rho_Y=-\frac12Y\rho_Y$ becomes:
$$-\int(\Delta u) Y\cdot\nabla u\rho_Y dY=\frac 14\int|\nabla u|^2 |Y|^2\rho_Y-\frac 12\int|Y\cdot\nabla u|^2\rho_Y-\int\rho_Y|\nabla u|^2.$$ 
Hence:
\bee
\left\|-\Delta u +\frac 12Y\cdot\nabla u\right\|_{L^2_{\rho_Y}}^2&=&\|\Delta u\|^2_{L^2_{\rho_Y}}+\frac 14\int \rho_Y(|Y|^2|\nabla u|^2-|Y\cdot\nabla u|^2)-\int\rho_Y |\nabla u|^2\\
&\geq & \|\Delta u\|^2_{L^2_{\rho_Y}}-\|\nabla  u\|_{L^2_{\rho_Y}}^2
\eee
which concludes the proof of \eqref{estimationapoinds}.
\end{proof}

\section{Proof of Lemma \ref{lemma:inversionoftheellipticsystemdefiningVij}}\label{sec:proofinversionoftheellipticsystemdefiningVij}

Recall that $j\in\mathbb{N}$ and $u_j(r)$ is the solution to
$$(\L_r+j)u_j=f_j\textrm{ and }(u_1, \Lambda_r\Phi)=0\textrm{ if }j=1,$$
where $f_j$ satisfies in the case $j=1$ 
$$(f_1,\Lambda_r\Phi)_{L^2_{\rho_r}}=0.$$
Recall also from Lemma \ref{lemma:spectrumoftheoperatormathcalLr} and \eqref{hypspectral} that $\L_r+j$ is a selfadjoint operator with domain $\mathcal{D}(\L_r)\subset L^2(r^2\rho_rdr)$ and
\bee
\textrm{Ker}(\L_r+1)=\la\Lambda_r\Phi\ra\textrm{ and Ker}(\L_r+j)=\{0\}\textrm{ for all }j\in\mathbb{N}\setminus\{1\}.
\eee
We immediately infer that we can solve uniquely  
$$(\L_r+j)u_j=f_j\textrm{ and }(u_1, \Lambda_r\Phi)=0\textrm{ if }j=1,$$
as long as $$f_j\in L^2_{\rho_r}\ \ \mbox{with}\ \ (f_1,\Lambda_r\Phi)_{L^2_{\rho_r}}=0\textrm{ in the case }j=1,$$ 
and there holds the following trivial bound for $k\in\mathbb{N}$ 
\bea\label{eq:trivialboundfollowingkernelofLrandorthogonality}
\sum_{l=0}^{k+2}\|\pa_r^lu_j\|_{L^2_{\rho_r}}\lesssim_k \sum_{l=0}^k\|\pa_r^lf_j\|_{L^2_{\rho_r}}\lesssim_k\sum_{l=0}^k\|\la r\ra^{\frac{2}{p-1}+l-\eta}\partial_r^lf_j\|_{L^{\infty}}.
\eea

Next, we derive a pointwise bound for derivatives of $u_j$ in the region $r\geq 1$. There exists two independent solutions $\varphi_{1,j}$ and $\varphi_{2,j}$ of
$$(\L_r+j)\varphi=0$$
smooth on $(0, +\infty)$ such that 
$$\varphi_{1,j}\sim r^{\frac{2}{p-1}+2j-3}e^{\frac{r^2}{4}},\,\,\,\, \varphi_{2,j}\sim \frac{1}{r^{\frac{2}{p-1}+2j}}\textrm{ as }r\to +\infty,$$
and their Wronskian 
$$W:=\varphi_{1,j}'\varphi_{2,j} - \varphi_{2,j}'\varphi_{1,j}$$
is given by
$$W=\frac{1}{r^2} e^{\frac{r^2}{4}}.$$
See for example Lemma 3.4 in \cite{CRS} for a proof. Then, using variation of constants as well as the estimates \eqref{eq:trivialboundfollowingkernelofLrandorthogonality} satisfied by $u_j$ implies that $u_j$ is given by
\bee
u_j(r) &=& \left(\int_r^{+\infty}f_j\varphi_{2,j}(r')^2e^{-\frac{(r')^2}{4}}dr'\right)\varphi_{1,j}(r)+\left(a_j - \int_1^rf_j\varphi_{1,j}(r')^2e^{-\frac{(r')^2}{4}}dr'\right)\varphi_{2,j}(r)
\eee
 where the constant $a_j$ is given by
 \bee
 a_j &=& \frac{1}{\varphi_{2,j}(1)}\left(u_j(1) - \left(\int_1^{+\infty}f_j\varphi_{2,j}r^2e^{-\frac{r^2}{4}}dr\right)\varphi_{1,j}(1)\right).
 \eee
 This immediately yields the pointwise bound\footnote{Note that we may take $\eta=0$ for $j\geq 1$. Only the case $j=0$ actually contains a log divergence and hence requires $\eta>0$.} 
 $$ \sum_{l=0}^k\|\la r\ra^{\frac{2}{p-1}+l-\eta}\partial_r^lu_j\|_{L^{\infty}(r\geq 1)}\lesssim_{k,\eta} |u_j(1)|+ \sum_{l=0}^k\|\la r\ra^{\frac{2}{p-1}+l-\eta}\partial_r^lf_j\|_{L^{\infty}}.$$ 

Finally, we derive a pointwise bound for $u$ in $r\leq 1$. By the Sobolev embedding in dimension 3 and \eqref{eq:trivialboundfollowingkernelofLrandorthogonality}, we have 
\bee
\sum_{l=0}^k \|\partial^l_ru_j\|_{L^{\infty}(r\leq 1)} &\lesssim& \sum_{l=0}^k\|\partial^l_ru_j\|_{H^2(r\leq 1)}\\
&\lesssim& \sum_{l=0}^{k+2}\|\pa_r^lu_j\|_{L^2_{\rho_r}}\\
&\lesssim_k& \sum_{l=0}^k\|\la r\ra^{\frac{2}{p-1}+l-\eta}\partial_r^lf_j\|_{L^{\infty}}.
\eee
This concludes the proof of Lemma \ref{lemma:inversionoftheellipticsystemdefiningVij}.
\end{appendix}

\subsection*{Funding}  P.R. and J.S are supported by the ERC-2014-CoG 646650 SingWave. F.M. is supported by the ERC Advanced grant BLOWDISOL 291214.

\subsection*{Conflict of Interest}  The authors declare that they have no conflict of interest.

\end{document}